\documentclass[a4paper,13pt]{article}
 \usepackage{amsmath}
 \usepackage{amssymb}
 \usepackage[margin=21mm]{geometry}
  \usepackage[utf8]{inputenc}
\usepackage{graphics}
\usepackage{xspace}
\usepackage{color}
  \usepackage{amsthm}
 \usepackage[old]{old-arrows}
 \usepackage{amsfonts}
 \usepackage[dvipsnames]{xcolor}
\usepackage{mathtools}
\usepackage{latexsym}
\newcommand{\R}{{\mathbb R}}

\newcommand{\N}{{\mathbb N}}

\newcommand{\dyle}{\displaystyle}
\newcommand{\dint}{\dyle\int}

\newtheorem{Theorem}{Theorem}[section]
\numberwithin{Theorem}{section}
\newtheorem{Corollary}[Theorem]{Corollary}

\newtheorem{Proposition}[Theorem]{Proposition}
\newtheorem{Definition}[Theorem]{Definition}

\newtheorem{remark}[Theorem]{Remark}
\usepackage[mathscr]{euscript}
\usepackage{mathrsfs}
 \usepackage{enumerate}
 \usepackage{cite}
 \usepackage{hyperref}
\newtheorem{theorem}{Theorem}[section]
\numberwithin{equation}{subsection}
\newtheorem{lemma}[theorem]{Lemma}

\newtheorem{definition}[theorem]{Definition}

\catcode`@=12

\makeatother
\usepackage[english]{babel}

\title{On singular problems associated with mixed operators under mixed  boundary conditions}
\date{}
\author{ Tuhina Mukherjee$^{1}$\thanks{Corresponding author}, Lovelesh Sharma$^{1}$  \\
     \small $^{1}$ Department of Mathematics, Indian Institute of Technology Jodhpur, Rajasthan 342030, India \\
         }
\newcommand{\Addresses}{{
  \bigskip
  \footnotesize
T.~Mukherjee, \textit{E-mail address:} \texttt{tuhina@iitj.ac.in}\\

 L. Sharma, \textit{E-mail address:} 
\texttt{sharma.94@iitj.ac.in, loveleshkaushik679@gmail.com}

}}
\begin{document}
\maketitle \vspace{-1.8\baselineskip}
\begin{abstract}
In this paper, we study the following singular problem associated with mixed operators (the combination of the classical Laplace operator and the fractional Laplace operator)  under mixed  boundary conditions
\begin{equation*} \label{1}
    \left\{
    \begin{aligned}
        \mathcal{L}u &= g(u), \quad u > 0 \quad \text{in} \quad \Omega, \\
        u &= 0 \quad \text{in} \quad U^c, \\
        \mathcal{N}_s(u) &= 0 \quad \text{in} \quad \mathcal{N}, \\
        \frac{\partial u}{\partial \nu} &= 0 \quad \text{in} \quad \partial \Omega \cap \overline{\mathcal{N}},
    \end{aligned}
    \right.
    \tag{$P_\lambda$}
\end{equation*}
 where  $U= (\Omega \cup {\mathcal{N}} \cup (\partial\Omega\cap\overline{\mathcal{N}}))$, $\Omega \subseteq \mathbb{R}^N$ is a non empty open set, $\mathcal{D}$, $\mathcal{N}$ are open subsets of $\mathbb{R}^N\setminus{\bar{\Omega }}$ such that ${\mathcal{D}} \cup {\mathcal{N}}= \mathbb{R}^N\setminus{\bar{\Omega}}$, $\mathcal{D} \cap {\mathcal{N}}=  \emptyset $ and $\Omega\cup \mathcal{N}$ is a bounded set with smooth boundary, $\lambda >0$ is a  real parameter and
 $\mathcal{L}=  -\Delta+(-\Delta)^{s},~ \text{for}~s \in  (0, 1).$
 Here $g(u)=u^{-q}$ or $g(u)= \lambda u^{-q}+ u^p$ with  $0<q<1<p\leq 2^*-1$.
We study $(P_\lambda)$ to derive the existence of weak solutions along with its $L^\infty$-regularity. Moreover, some Sobolev-type variational inequalities associated with these weak solutions are established. 
 \end{abstract}

 \textbf{Keywords:} Mixed local-nonlocal operators, mixed boundary conditions,  singular and critical nonlinearity, existence and uniqueness results, regularity.
 
 \textbf{Mathematics Subject Classification:} 35A01, 35A15, 35R11, 35J25, 35J20,  47G20.

 
\section{Introduction}
In this paper, we shall study the existence, uniqueness, and regularity of the following singular problem associated with mixed local and nonlocal operator
\begin{equation*}
    \left\{
    \begin{aligned}
        \mathcal{L}u &= g(u), \quad u > 0 \quad \text{in} \quad \Omega, \\
        u &= 0 \quad \text{in} \quad U^c, \\
        \mathcal{N}_s(u) &= 0 \quad \text{in} \quad \mathcal{N}, \\
        \frac{\partial u}{\partial \nu} &= 0 \quad \text{in} \quad \partial \Omega \cap \overline{\mathcal{N}},
    \end{aligned}
    \right.
    \tag{$P_\lambda$}
\end{equation*}
  where $U= (\Omega \cup {\mathcal{N}} \cup (\partial\Omega\cap\overline{\mathcal{N}}))$, $\Omega \subseteq \mathbb{R}^N$, $\mathcal{D}$, $\mathcal{N}$, respectively denote open Dirichlet and Neumann set, are open subsets forming a partition of $\mathbb{R}^N\setminus{\bar{\Omega }}$ 
  	and $\Omega\cup \mathcal{N}$ is a bounded set with smooth boundary, $\lambda >0$ is a  real parameter, and
 \begin{equation}\label{A1}
\mathcal{L}=  -\Delta+(-\Delta)^{s},~ \text{for}~s \in  (0, 1).
 \end{equation}
Here $g(u)=u^{-q}$ or $g(u)= \lambda u^{-q}+ u^p$ with  $0<q<1< p\leq 2^*-1$. The term "mixed" describes an operator which is a combination of both local and nonlocal differential operators. In our case, the operator $\mathcal{L}$ in \eqref{1} is generated by the superposition of the classical Laplace operator $-\Delta$ and the fractional Laplace operator, $(-\Delta)^{s}$ which is for a fixed parameter $s \in (0,1)$  defined by
$${(- \Delta)^{s}u(x)} = C_{n,s}~ P.V. \int_{\mathbb{R}^N} {\dfrac{u(x)-u(y)}{|x-y|^{N+2s}}} ~ dy. $$ 
The term "P.V." stands for Cauchy's principal value, while $C_{N
,s}$ is a normalizing constant whose explicit expression is given by
 $$C_{N,s}= \bigg( \int_{\mathbb{R}^{N}} {\frac{1-cos({\zeta}_{1})}{|\zeta|^{N+2s}}}~d \zeta \bigg)^{-1}.$$ 
In the literature, there are numerous definitions of nonlocal normal derivatives. We consider the one suggested in \cite{MR3651008} for smooth functions $u$ as
\begin{equation}\label{normal}
\mathcal{N}_{s}u(x):=\dint_{\Omega} \dfrac{u(x)-u(y)}{|x-y|^{N+2s}}\,dy, \qquad  x\in \mathbb{R}^N\setminus\bar{\Omega}.
\end{equation}
To gain a better understanding of nonlocal operators and their applications in real-world situations, we recommend reading the Hitchhiker's Guide \cite{MR2944369} and its references. The study of mixed operators of the form $\mathcal{L}$ in the problem \eqref{1} is motivated by several applications, including optimum searching, biomathematics, and animal forging; we refer to \cite{dipierro2022non, MR4249816, MR2924452, MR3590678, MR3771424}. Heat transmission in magnetic polymers is one of the other popular applications; see \cite{blazevski2013local}. In the applied sciences, bi-modal power law distribution systems, as well as models derived from combining two distinct scaled stochastic processes, naturally develop these types of operators. For an in-depth explanation, see \cite{MR4225516, MR1640881} and its references. 
To this aim, we start with a brief background of the problems available in the literature. Over the past few decades, a lot of research has been done on singular elliptic problems.
 Let us consider elliptic problems consisting of singular nonlinearity or blow-up nonlinearity and critical exponent term, especially of the type
 \begin{equation}\label{singular}
\begin{cases}
          -\Delta u =  \lambda u^{-q}+ u^r,\quad u>0\quad
 \text{in }\Omega,\\
u = 0 \quad \text{on}~ \partial\Omega,\end{cases}
\end{equation}
 where $q>0$ and $r\in (1,\frac{N+2}{N-2}]$.
Crandall et al. pioneering research  \cite{RA} established that the unperturbed and local case of \eqref{1} under Dirichlet boundary conditions provided by
\begin{equation*}
\begin{cases}
-\Delta u = \frac{f}{u^{\gamma}}, & u > 0 \quad \text{in } \Omega, \\
\quad ~u = 0 & \text{in } \partial\Omega.
\end{cases}
\end{equation*}
admits a unique solution \( u \in C^2(\Omega) \cap C(\overline{\Omega}) \) for any, \( \gamma > 0 \) along with the fact that the solution behaves like a distance function near the boundary provided \( f \) is H\"{o}lder continuous. Due to the presence of a singular term, the associated energy
functional is no longer $C^1$, therefore, the standard critical point theory fails, which makes such problems very interesting. For $q\in (0,1)$, Haitao et al. in \cite{MR1964476} studied \eqref{singular} with the sub-super solution techniques, whereas authors in \cite{MR2099611} studied the same problem \eqref{singular} via the Nehari manifold technique. Both articles establish the existence of at least two solutions and regularity results. The article \cite{MR2446183} deals with the delicate $q\geq1$ case, where again the existence of at least two weak solutions were established using the advanced theory of nonsmooth analysis. 
Lazer et al. \cite{LA} showed that unique solution found by  \cite{RA} is in $W_0^{1,2}(\Omega)$ if and only if \( 0 < \gamma < 3 \). Furthermore, they demonstrated that the answer is in \(C^1(\overline{\Omega}) \) as long as \(0 < \gamma < 1 \).
 In \cite{Barrios} the authors studied the following singular problem
 \begin{equation}\label{A}
\begin{cases}
(-\Delta)^s u = \lambda \frac{f(x)}{u^{\gamma}} + M u^p, & u > 0 \quad \text{in } \Omega, \\
~~\quad \quad ~u = 0 & \text{in } \mathbb{R}^N \setminus \Omega.
\end{cases}
\end{equation}
where $ N > 2s $, $ M \geq 0 $, $ 0 < s < 1 $, $ \gamma > 0 $, $ \lambda > 0 $, $ 1 < p < 2_s^* - 1 $, and $ f \in L^m(\Omega) $, $ m \geq 1 $ is a nonnegative function. The authors investigated weak solutions using uniform estimates of sequence \(\{u_n\}\), which are regularised problem solutions with the singular term \( u^{-\gamma} \) substituted by \( (u + \frac{1}{n})^{-\gamma} \). They also addressed multiplicity results for $M > 0$ and small \( \lambda \) in the subcritical case. 
For $ M =0$ with the classical Laplace operator, the semilinear elliptic problem in equation \eqref{A} corresponds to the problem studied by Boccardo et al. in \cite{Bo}.

 In recent years, there has been a considerable focus on investigating elliptic problems involving mixed-type operators $\mathcal{L}$, as in \eqref{1}, exhibiting both local and nonlocal behaviour. 
Let us put some light on the literature concerning problems involving the mixed operator $\mathcal{L}$, among the existing long list. Cassani et al. in \cite{Cassani} examined the spectral properties, existence and nonexistence results, as well as regularity for solutions to a one-parameter family of elliptic equations involving mixed local and nonlocal operators. Biagi et al. in \cite{BS} conducted a thorough analysis of a mixed local and nonlocal elliptic problem. Their research established the existence of solutions, explored maximum principles that dictate the behaviour of these solutions, and investigated the interior Sobolev regularity of the solutions. Their findings enhance our understanding of the properties and behaviour of solutions to the following type of problem
\begin{equation}\label{B}
\begin{cases}
\mathcal{L}u = g(u), & u > 0 \quad \text{in } \Omega, \\
u ~~= 0 & \text{in } \mathbb{R}^N \setminus \Omega.
\end{cases}
\end{equation}

They also examined various interesting inequalities related to the mixed operator $\mathcal{L}$ in \cite{MR4391102}. Recently, Bal et al. have studied the following  mixed local-nonlocal elliptic singular problem,
\begin{equation}\label{C}
\begin{cases}
 -\Delta_p u + (-\Delta)^s_p u = \frac{\lambda}{u^{\gamma}} + u^r, & u > 0 \quad \text{in } \Omega, \\
~~\quad \quad ~u = 0, & \text{in } \mathbb{R}^N\setminus \Omega
\end{cases}
\end{equation}
where, $\Omega \subset \mathbb{R}^N$ is a bounded domain with a smooth boundary, $p > 1$, $N > p$, $s \in (0, 1)$, $r \in (p - 1, p^* - 1)$, and $p^*$ is the critical Sobolev exponent. The authors established the existence of two weak solutions to problem \eqref{C} for $0 < \gamma < 1$ and certain values of $\lambda$. Furthermore, they showed that for any $\gamma > 0$ and $p = 2$ and $s \in (0, 1/2)$, there are at least two positive weak solutions to problem \eqref{C} for small values of $\lambda$.
We also recall \cite{arora2023combined, MR4444761}, where the authors explored purely singular problems with mixed operators along with Dirichlet boundary conditions.
 They obtained various results concerning the existence and other properties of solutions. In \cite{PG}, using the variational methods, authors showed that there are at least two solutions to \eqref{C} for $\gamma < 1$ and $p = 2$.  Lamao et al. in \cite{MR4357939} focused on the behaviour and properties of solutions to problem \eqref{B}, specifically analyzing their summability characteristics. Additionally, Arora et al. \cite{arora2021combined} investigated the existence and non-existence results for both singular and non-singular problems of the form \eqref{B}. 
 
 We also refer interested readers to \cite{MR3445279}, which contains a study of the Neumann problem with a mixed operator. Dipierro et al. \cite{MR4438596} were the first to consider mixed operator problems with classical and nonlocal Neumann boundary conditions. Their recent article discusses the spectral properties and $L^\infty$ bounds related to a mixed local and nonlocal problem, with specific applications arising from population dynamics and mathematical biology.
Inspired by the above literature, we in \cite{Mukherjee} were the first to study elliptic problems involving mixed operators under mixed Dirichlet-Neumann boundary conditions with a concave-convex type nonlinearity. We established results on the existence, nonexistence, and multiplicity of positive solutions, along with proving Picone identity and the maximum principle. Moreover, Giacomoni et al. \cite{JG} analyzed an eigenvalue problem involving mixed operators under mixed boundary conditions, providing bifurcation results from both zero and infinity for an asymptotically linear problem.

Notably, no work address mixed operators with mixed Dirichlet-Neumann boundary conditions involving singular nonlinearity. This gap piqued our curiosity about the implications of formulating a PDE that incorporates a mixed operator $\mathcal{L}$ under boundary conditions that combine Dirichlet data in some regions and Neumann (both local and nonlocal) data in others. Our paper is distinguished by its innovative approach to this question, which involves purely singular or singular with subcritical and critical  exponent type nonlinearity (see \eqref{1}) and the combination of mixed operator as well as mixed Dirichlet-Neumann boundary, which is the striking feature of our paper. Moreover, we established a mixed Sobolev inequality and demonstrated its relation to the solution of the mixed purely singular problem (see Theorem \ref{thm2.4}). Additionally, we proved the weak comparison principle along with other complementary properties and regularity results, all of which constitute the novel contributions of this article. Our problem lacks compactness in the sense of Sobolev embedding due to an additional critical exponent term, which makes \eqref{1} difficult as well as interesting. To overcome this, we seek the help of the Brezis-Lieb lemma.  To establish the existence of a second solution in the case of singular critical type nonlinearity, sharp estimates for the minimizers of the best constant will be required, which remains an open question for future research.

\vspace{0.2cm}
\textbf{Organization of the Article:}
The structure of our article is as follows: Section 2 provides an introduction to the functional framework required to address the problem \eqref{1}. It presents the specific notion of the solution that will be employed and introduces relevant auxiliary and preliminary results.
Section 3 focuses on establishing the existence and uniqueness of solutions and $L^{\infty}$ estimates for purely singular problems \eqref{P_q} for $0<q\leq 1$. Authors employ an approximation method by
modifying (or truncating) the singular term. Also, we discussed the complementary properties of the sequence of solutions to \eqref{P_q} and established mixed Sobolev-type inequality. In Section 4, the authors studied perturbed singular problems, using Nehari manifold analysis and variational techniques to establish existence results. At the end of this section, using the Moser iteration method, authors proved the regularity of the solution, i.e., $L^{\infty}(U)$. 
\vspace{0.2cm}

\textbf{Notations}

In this paper, we will adopt the following notations:
\begin{itemize}
    \item For $u: \R^N \to \R$ measurable functions, we denote by $ u^+ = \max\{u, 0\}$ and $ u^- = \max\{-u, 0\}$.
    \item For $k > 1$, we denote by $k' = \frac{k}{k-1}$ the conjugate exponent of $ k$.
    \item The constant \( C \) may vary between lines or within the same line. If constant $ C$ depends on parameters such as $k_1, k_2, \ldots $, it is expressed as $ C = C(k_1, k_2, \ldots)$.
\end{itemize}

\section{Function space and main results}
In this section, we have set our notations and formulated the functional setting for \eqref{1}, which is used throughout the paper. Also, we state the main results after stating some preliminaries.
For every $s\in (0,1)$, we recall the fractional Sobolev spaces
$${H^{s}(\mathbb{R}^N)} =  \Bigg\{ u \in L^{2}(\mathbb{R}^N):~~\frac{|u(x) - u(y)|}{|x - y|^{\frac{N}{2} + s}} \in L^{2}({\mathbb{R}^N}\times {\mathbb{R}^N)} \Bigg\} $$ {which contains $H^1(\mathbb R^N)$.} We assume that $\Omega \cup \mathcal N$ is bounded with smooth boundary. 
The symbol $U$ denotes $(\Omega \cup {\mathcal{N}} \cup (\partial\Omega\cap\overline{\mathcal{N}}))$ to keep things simple.
We define the function space $\mathcal{X}^{1,2}_{\mathcal{D}}(U)$ as 
\[
    \mathcal{X}^{1,2}_{\mathcal{D}}(U)  = \{u\in H^1(\mathbb{R}^N) : ~u|_{U} \in H^1_0(U) ~\text{and}~ u \equiv 0~ a.e. ~\text{in}~ {U^c}\}.
\]
Let us define 
 $$ \eta(u)^2 =  ||\nabla{u}||^2_{L^{2}(\Omega)}+ [u]^2_{s},$$
for $u\in\mathcal{X}^{1,2}_{\mathcal{D}}(U)$, where  $[u]_s$ is the Gagliardo seminorm of $u$ defined by 
 $$[u]^2_{s} = ~ \bigg(\int_{Q} \frac{|u(x)-u(y)|^{2}}{|x-y|^{N+2s}} \, dx dy \bigg)$$ and $Q= \mathbb R^{2N}\setminus (\Omega^c\times \Omega^c)$.
 Moreover, we define the local function space as
$$
\mathcal{X}^{1,2}_{\mathcal{D}, \text{loc}}(U) = \{ u : \R^N \to \R :~ \nabla u \in L^2(K),~ [u]^2_{s,K}< +\infty, ~\forall~ K \Subset \Omega\}
$$
where
$$
[u]^2_{s,K} = \int_{K\times K} \frac{|u(x) - u(y)|^2}{|x - y|^{N + 2s}} \, dx \, dy, ~ \forall~K\Subset \Omega.$$ 
 Here $ K \Subset \Omega$  means $ K $ is compactly contained in $ \Omega $, thus $K\times K$ must be compactly contained within  $Q$.

The following Poincar\'e type inequality can be established following the arguments of [\cite{MR4065090}, Proposition 2.4]   and taking advantage of partial Dirichlet boundary conditions in $U^c$.

\begin{Proposition}\label{Poin} (Poincar\'e type inequality) There exists a constant $C=C(\Omega, N,s)>0$ such that
$$
\dint_{\Omega}| u|^2\,dx\leq C\bigg(\int_{\Omega} |\nabla u|^2\,dx+  \int_{Q} \dfrac{|u(x)-u(y)|^2}{|x-y|^{N+2s}}\,dxdy\bigg),
$$
for every  $u\in\mathcal{X}^{1,2}_{\mathcal{D}}(U)$, i.e. $\|u\|^2_{L^2(\Omega)}\leq C \eta(u)^2$. 
\end{Proposition}
As a consequence of Proposition \ref{Poin}, $\eta(\cdot)$ forms a norm on $\mathcal{X}^{1,2}_{\mathcal{D}}(U)$ and $\mathcal{X}^{1,2}_{\mathcal{D}}(U)$ is a Hilbert space with the  inner product associated with $\eta(\cdot)$ i.e.
$$\langle{ u},{ v}\rangle = \int_{\Omega} \nabla u. \nabla{v} \,dx + \int_{Q} {\dfrac{(u(x)-u(y)) (v(x)-v(y))}{|x-y|^{N+2s}}} ~ dx dy. $$

 \begin{Proposition}\label{P}
 For every $ u,v\in  C^\infty_0(U)$, it holds
\begin{align*}
    \int_{\Omega}v \mathcal{L} u \,dx  
    &= \int_{\Omega} \nabla u \cdot\nabla{v} \,dx +  \int_{Q} {\dfrac{(u(x)-u(y)) (v(x)-v(y))}{|x-y|^{N+2s}}} ~ dx dy-  \int_{\partial \Omega\cap\overline{{\mathcal{N}}}} v {\frac{\partial u}{\partial \nu}}~ d{\sigma}-  \int_{{\mathcal{N}}} v {\mathcal{N}}_s u~ dx.
    \end{align*}
    where $\nu$ denotes the outward normal on $\partial\Omega$.
    \end{Proposition}
    \begin{proof}
        By directly using the integration by parts formula and the fact that $u,v \equiv 0$ a.e. in $\mathcal{D}\cup(\partial\Omega\cap\overline{\mathcal{D}})=U^c$, we can follow Lemma 3.3 of \cite{MR3651008}, to obtain the conclusion.
    \end{proof}
  \begin{Corollary}
Since $ C^\infty_0(U)$ is dense in $\mathcal{X}^{1,2}_{\mathcal{D}}(U)$, so Proposition \ref{P} still holds for functions in $\mathcal{X}^{1,2}_{\mathcal{D}}(U)$.   
\end{Corollary}

\begin{Definition}\label{d1}
We say that $u\in \mathcal{X}^{1,2}_{\mathcal{D}}(U)$ is a weak solution to the problem \eqref{1} if, $u>0$ a.e., in $\Omega$ and  
\begin{equation}\label{dd2.2}
    \int_{\Omega} \nabla u.\nabla \varphi \,dx +\int_{Q} \frac{(u(x)-u(y))(\varphi(x)-\varphi(y))}{|x-y|^{N+2s}} dxdy  = \int_{\Omega} g(u) \varphi ~dx, ~\forall~ \varphi \in C^\infty_0(\Omega). 
\end{equation}
\end{Definition}

For functions in $\mathcal{X}^{1,2}_{\mathcal{D}}(U)$, we now present an embedding result that is a consequence of $\mathcal{X}^{1,2}_{\mathcal{D}}(U)\hookrightarrow H^1(\mathbb R^N)$ and Sobolev embedding of $H^1(\mathbb R^N)$, see \cite{SK}. 

\begin{remark}\label{r2.5}
   Since $U$ is bounded (because $\Omega\cup\mathcal{N}$ is bounded) with smooth boundary, so we have the compact embedding
    $$\mathcal{X}^{1,2}_{\mathcal{D}}(U)\hookrightarrow \hookrightarrow L^r_{loc}(\R^N)$$
   for $r\in [1,2^*)$ and continuous embedding for $r\in[1, 2^*],$ where $2^*= \frac{2N}{N-2}$ with $N>2.$
\end{remark}
\subsection{Preliminaries}
We present an existence and regularity result for an auxiliary problem that is crucial to our purely singular problem addressed in the subsequent section.
\begin{lemma}\label{l3.1}
 Let $f \in L^{\infty}(\Omega) \setminus \{0\}$ be nonnegative. Then the problem
 \begin{equation}\label{l31}
\begin{cases}  
     \mathcal{L}u &= ~f(x), \quad u>0 \quad 
 \text{in }\Omega,\\ 
  u&=~~0~~\text{in} ~~U^c,\\
 \mathcal{N}_s(u)&=~~0 ~~\text{in} ~~{{\mathcal{N}}}, \\
 \frac{\partial u}{\partial \nu}&=~~0 ~~\text{in}~~ \partial \Omega \cap \overline{{\mathcal{N}}},
      \end{cases}
    \end{equation}
 admits a unique solution $u \in \mathcal{X}^{1,2}_{\mathcal{D}}(U)$. Moreover, for some constant $\beta=\beta(\omega)>0,$ 
\begin{equation}\label{strict-pos}
u(x) \geq \beta >0,~\text {for a.e. }~ x \in  \omega \Subset \Omega.
\end{equation}
    \end{lemma}
\begin{proof}
In order to prove the existence of a solution to \eqref{l31}, we define the energy functional $\mathcal{J}: \mathcal{X}^{1,2}_{\mathcal{D}}(U)\to \mathbb{R}$ by
\begin{equation}
\mathcal{J}(u) := \frac{1}{2} \int_{\Omega} |\nabla u|^2 \, dx + \frac{1}{2} \int_{Q}  \frac{|u(x) - u(y)|^2}{|x - y|^{N + 2s}} \, dx \, dy - \int_{\Omega} f(x) u \, dx, ~~ \forall~u\in\mathcal{X}^{1,2}_{\mathcal{D}}(U),
\end{equation}
and then we look for a solution of \eqref{l31} as a critical point to function $\mathcal{J}.$ The following properties are easy to verify:
\begin{enumerate}
    \item  $\mathcal{J}$ is coercive. Since, by Remark \ref{r2.5} it follows
\begin{equation*}
\mathcal{J}(u) \geq\frac{1}{2}\eta(u)^2 - a |\Omega|^{\frac{1}{2}}  \| f \|_{L^{\infty}(\Omega)} \eta(u),
\end{equation*}
where $a> 0$ is the embedding constant. Then $\mathcal{J}(u)\to +\infty$ as $\eta(u)\to+\infty$. 
\item $\mathcal{J}$ is weakly  lower semi-continuous in $\mathcal{X}^{1,2}_{\mathcal{D}}(U),$ since functional $\mathcal{J}$ is convex and $C^1$ functional.
\end{enumerate}
Thus, we conclude that $\mathcal{J}$ has a minimizer on $\mathcal{X}^{1,2}_{\mathcal{D}}(U)$, which solves \eqref{l31}.
Now, we claim that the solution of \eqref{l31} obtained as minimizer is unique. Suppose $u_1, u_2 \in \mathcal{X}^{1,2}_{\mathcal{D}}(U)$ are  two distinct solutions to \eqref{l31} and subtracting them then for every $\varphi \in \mathcal{X}^{1,2}_{\mathcal{D}}(U)$, we have,
\begin{equation}\label{e305}
\int_{\Omega} \big( \nabla u_1 -  \nabla u_2 \big)\cdot \nabla \varphi \, dx 
+ \int_{Q} {\dfrac{\big((u_1 - u_2)(x) - (u_1 - u_2)(y)\big) (\varphi(x)-\varphi(y))}{|x-y|^{N+2s}}} ~ dx dy  \,  = 0.
\end{equation}
Putting $\varphi= u_1 - u_2$ in \eqref{e305}, we get, $\eta(u_1-u_2)=0$
which clearly says that $u_1-u_2=0$ a.e. in $\R^N$.
Hence, the solution of \eqref{l31} is unique.
Now we show that the solution of \eqref{l31} is nonnegative.  Taking $u^-=\max\{-u,0\}$ as a test function in \eqref{l31}, it follows that 
\begin{equation}\label{eq2.18}
         0 \leq\int_{\Omega} f u^- \, dx= \int_{\Omega} \nabla u \cdot\nabla{u^-} \,dx +  \int_{Q} {\dfrac{(u(x)-u(y)) (u^-(x)-u^-(y))}{|x-y|^{N+2s}}} ~ dx dy.
        \end{equation}
It is easy to see that
$$
\begin{aligned}
    (u(x)-u(y)) (u^-(x)-u^-(y))
    &= \left[\left(u^+(x)-u^+(y)\right)-\left(u^-(x)-u^-(y)\right)\right]\left(u^-(x)-u^-(y)\right)\\
    &= -u^+(x)u^-(y)-u^+(y)u^-(x)-\left|\left(u^-(x)-u^-(y)\right)\right|^2\leq -\left|\left(u^-(x)-u^-(y)\right)\right|^2
    \end{aligned}
    $$
and $\nabla u \cdot\nabla{u^-} = -|\nabla u^-|^2$ then from \eqref{eq2.18} we deduce that
{$$0 \leq\int_{\Omega} fu^- \, dx\leq  \int_{\Omega} -|\nabla{u^-}|^2 \,dx -  \int_{Q} {\dfrac{ \left|\left(u^-(x)-u^-(y)\right)\right|^2}{|x-y|^{N+2s}}} ~ dx dy\leq 0.$$}
Thus,
$0\leq -\eta(u^-)^2\leq 0$ which says $u^-=0$ a.e. in $\R^N.$
Hence, $u\geq 0$ a.e. in $\R^N$. 
Concerning the second assertion, as $f\not\equiv 0$, we have $u\not\equiv 0$ in $U.$ Thus, utilizing [\cite{Kinnunen}, Theorem 8.3], we obtain for some constant $\beta=\beta(\omega)>0,$ 
\begin{equation}
u(x) \geq \beta >0,~\text {for a.e. }~ x \in  \omega \Subset \Omega
\end{equation}
which finishes the proof.
\end{proof}

\begin{lemma}\label{linfty}
Let $f\in L^{\infty}(U)$  and  $u\in \mathcal{X}^{1,2}_{\mathcal{D}}(U) $ is a weak solution of \eqref{l31} then $u\in L^{\infty}(U)$. 
\end{lemma}
\begin{proof}
     We define $\mathcal{A}(k)=\{x \in \Omega: ~|u(x)| \geq k\}$, for any $k>0$.
  Choosing
 $$
\varphi_k(x) = (\operatorname{sgn} u) \max(|u| - k, 0) =
\begin{cases} 
   u - k, & \text{if } u \geq k, \\[2mm]
   0, & \text{if } |u| \leq k, \\[2mm]
   u + k, & \text{if } u \leq -k.
\end{cases}
$$
  as a test function in \eqref{l31}, we have
\begin{equation}\label{bieq}
\int_{\Omega} \nabla u \cdot \nabla \varphi_k \,d x+\int_{Q}  \frac{(u(x)-u(y))\left(\varphi_k(x)-\varphi_k(y)\right)}{|x-y|^{n+2 s}} d x d y= \int_{\Omega} f \varphi_k \,d x.
    \end{equation}
Hence, by non-negativity of the second integral in \eqref{bieq} 
 and  using the Sobolev embedding and H\'older inequality, we obtain
$$
\begin{aligned}
\int_{\Omega}\left|\nabla \varphi_k\right|^2 d x=\int_{\Omega} \nabla u \cdot \nabla \varphi_k d x & \leq \int_{\Omega} f \varphi_k d x \leq \int_{\mathcal{A}(k)} |f\varphi_k| \,d x  \leq C_0\|f\|_{L^{\infty}(\Omega)}|\mathcal{A}(k)|^{\frac{r-1}{r}}\left(\int_{\Omega}\left|\nabla \varphi_k\right|^2 d x\right)^{\frac{1}{2}},
\end{aligned}
$$
where $C_0$ is the Sobolev constant and $2<r\leq 2^*$. Hence, we have
\begin{equation}\label{515eq}
    \int_{\Omega}\left|\nabla \varphi_k\right|^2 d x \leq C_1\|f\|_{L^{\infty}(\Omega)}^2|A(k)|^{\frac{2(r-1)}{r}}.
\end{equation}
It is easy to verify that if we choose $h$ such that, $0<k<h$ then $|u(x)|-k\geq (h-k)$ on $A(h)$ and $\mathcal{A}(h) \subset \mathcal{A}(k)$. Using this fact and \eqref{515eq}, we find
$$
\begin{aligned}
(h-k)^2|\mathcal{A}(h)|^{\frac{2}{r}} &\leq\left(\int_{\mathcal{A}(h)}(|u(x)|-k)^r d x\right)^{\frac{2}{r}}  \leq\left(\int_{\mathcal{A}(k)}(|u(x)|-k)^r d x\right)^{\frac{2}{r}} \\
& \leq C_2\int_{\Omega}\left|\nabla \varphi_k\right|^2 d x 
\leq C\|f\|_{L^{\infty}(\Omega)}^2|\mathcal{A}(k)|^{\frac{2(r-1)}{r}}, \text{where}~ C_1C_2= C.
\end{aligned}
$$
Therefore, we have $|\mathcal{A}(h)| \leq C \frac{\|f\|_{L^{\infty}(\Omega)}^r}{(h-k)^r}|\mathcal{A}(k)|^{r-1}$, $\forall$ $h>k>0$. Thus using [\cite{MR3393266}, Lemma 14] or [\cite{Stampacchia}, Lemma B.1], we obtain $|\mathcal{A}(d)|=0$, where $d^r=c \|f\|_{L^{\infty}(\Omega)}^r 2^{\frac{r^2}{r-1}}$, and $0<c=c(r,\Omega, C_1,C_2)$. 
Hence, \begin{equation}\label{stp1}
\|u\|_{L^{\infty}(\Omega)} \leq C', \
    \end{equation}
    for some positive constant $C'=C'(C, \|f\|_{L^{\infty}(\Omega)})$. Thus, $u\in L^{\infty}(\Omega).$ {  Now,  if $x\in \mathcal{N}$, then using the definition of $\mathcal{N}_s$ (see  \eqref{normal}) and $\mathcal{N}_s u(x)=0$, we have
$$ u(x)\int_{\Omega} \frac{dy}{|x-y|^{n+2s}}= \int_{\Omega} \frac{u(y) dy}{|x-y|^{n+2s}}\implies
u(x)=\frac{\int_{\Omega} \frac{u(y) dy}{|x-y|^{n+2s}}}{\int_{\Omega} \frac{dy}{|x-y|^{n+2s}}}.$$
Since $u\in L^{\infty}(\Omega)$ then we get $|u(x)|\leq \|u\|_{L^{\infty}(\Omega)}$, for each $x\in\mathcal{N}.$ Thus, we conclude that $u\in L^{\infty}(U).$
}
\end{proof}
\subsection{Main results}

The theorems stated below are the main results of this article.
\begin{theorem}\label{th2}
    There exists a {unique} solution of $(P_\lambda)$ when $g(u)=u^{-q}, ~q>0.$
\end{theorem}
\begin{theorem}\label{thm2.4}
 Assume $ 0 < q < 1$. The mixed Sobolev inequality 
 \begin{equation}\label{eqmixed}
    \eta(v)^2\geq C\left(\int_{\Omega}|v|^{1-q}\,dx\right)^{\frac{2}{1-q}}
 \end{equation}
 is satisfied for every $ v \in \mathcal{X}^{1,2}_{\mathcal{D}}(U) $ if and only if the problem \eqref{P_q} admits a weak solution in  $\mathcal{X}^{1,2}_{\mathcal{D}}(U)$.

\end{theorem}
\begin{theorem}\label{2} Consider $(P_\lambda)$ with $g(u)= \lambda u^{-q}+u^p$  where $0<q<1<p\leq 2^*-1$, then 
 \begin{enumerate}
     \item There exists a threshold $\Lambda>0$
 such that $(P_\lambda)$ possesses at least two weak solutions $u_\lambda^*$
and $v_\lambda^*$ in  $\mathcal{X}^{1,2}_{\mathcal{D}}(U)$, whenever $\lambda \in (0, \Lambda)$ and $p\in (1,2^*-1)$,
\item $(P_\lambda)$ possesses atleast one weak solution for each $\lambda>0$ when $p=2^*-1$.
 \end{enumerate} 
 \end{theorem}
 \begin{theorem}\label{3}
      Any weak solution of $(P_\lambda)$ belongs to $L^\infty(U)$.
 \end{theorem}

\section{Purely singular problem }
This section contains proof of the main Theorem \ref{th2} and \ref{thm2.4}. The idea is to employ an approximation method by modifying(or truncating) the singular term $\frac{1}{u^q}$ so that we remove its singularity at the origin. We will then analyze the behaviour of a sequence $\{u_n\}$ which represents the solutions of the approximated problem $\eqref{P_nq}$, defined below.
Let us consider $(P_\lambda)$ with $g(u)=u^{-q}$ that is
\begin{equation*}\label{P_q}
\begin{cases}\tag{$P_q$}   
     \mathcal{L}u &= ~{u^{-q}}, \quad u>0 \quad 
 \text{in }\Omega,\\ 
  u&=~~0~~\text{in} ~~U^c,\\
 \mathcal{N}_s(u)&=~~0 ~~\text{in} ~~{{\mathcal{N}}}, \\
 \frac{\partial u}{\partial \nu}&=~~0 ~~\text{in}~~ \partial \Omega \cap \overline{{\mathcal{N}}},
      \end{cases}
    \end{equation*}
where $0<q\leq 1,$ and other notations mean same as defined before.

\textbf{Approximation to \eqref{P_q}:}
We define the sequence $\{u_n\}$ which are solutions to the following approximated problem 
\begin{equation*}\label{P_nq}
\begin{cases}\tag{$P_{n,q}$}
   \mathcal{L}u_n &=~  \frac{1}{\left(u_n+\frac{1}{n}\right)^q}, ~u_n>0 ~\text{in}~\Omega, \\ 
  u_n&=~0~~\text{in} ~~U^c,\\
 \mathcal{N}_s(u_n)&=~0 ~~\text{in} ~~{{\mathcal{N}}}, \\
 \frac{\partial u_n}{\partial \nu}&=~0 ~~\text{in}~~ \partial \Omega \cap \overline{{\mathcal{N}}}.
      \end{cases} 
    \end{equation*}
Firstly, we establish the existence and uniqueness of a solution to nonsingular problem \eqref{P_nq}.
\begin{lemma}\label{lem4.1}
For any $n \in \mathbb{N}$,   problem \eqref{P_nq} has a unique  solution $u_n \in \mathcal{X}^{1,2}_{\mathcal{D}}(U)\cap L^{\infty}(U)$.    
\end{lemma}
\begin{proof}
 We fix $n \in \mathbb{N}$ and for any $v \in L^2(U)$,  Lemma \ref{l3.1} and Lemma \ref{linfty} guarantees the existence of a unique solution $w\in \mathcal{X}^{1,2}_{\mathcal{D}}(U)\cap L^{\infty}(U)$ to the following problem
\begin{equation}\label{eq401}
\begin{cases}
\mathcal{L} w & =~\frac{1}{\left(v^{+}+\frac{1}{n}\right)^q}, ~ w>0 \quad \text{in }\Omega,\\
w&=~0~~\text{in} ~~U^c,\\
 \mathcal{N}_s(w)&=~0 ~~\text{in} ~~{{\mathcal{N}}}, \\
 \frac{\partial w}{\partial \nu}&=~0 ~~\text{in}~~ \partial \Omega \cap \overline{{\mathcal{N}}},
    \end{cases}
    \end{equation}
 where $v^+=\max\{v,0\}$. We define the map $L^2(U)\ni v \mapsto w=T(v)\in \mathcal{X}^{1,2}_{\mathcal{D}}(U)\subset L^2(U)$, where $w$ is a unique solution to \eqref{eq401}.
 Testing \eqref{eq401} with $w$ and using Remark \ref{r2.5} then  we get
\begin{equation}\label{eal}
    \begin{aligned}
\int_{\Omega} |\nabla w(x)|^2\,dx+\int_{Q }\frac{|w(x)-w(y)|^2}{|x-y|^{N+2 s}} d x d y =\int_{\Omega} \frac{ w}{\left(v^{+}+\frac{1}{n}\right)^q}\,dx \leq \int_{\Omega} \frac{|w|}{\left(\frac{1}{n}\right)^q}\, dx \leq n^{q}\|w\|_{L^1(U)},
       \end{aligned}
\end{equation}
where we have used the fact that $v^+\geq 0$ a.e. in $\R^N$ and $v^+ +\frac{1}{n}\geq \frac{1}{n}$.
Now using H\"{o}lder inequality and Remark \ref{r2.5}, there exists some suitable constant  $C>0$ such that
$$
\begin{aligned}
   \|w\|_{L^1(U)}&\leq |U|^{\frac{N+2}{2N}}\|w\|_{L^{2^*} (U)} \leq C|U|^{\frac{N+2}{2N}} \eta(w).  
    \end{aligned}
$$
So, from \eqref{eal}, we conclude that 
$$
\eta(w)^2  \leq C n^q|U|^{\frac{N+2}{2N}} \eta(w).
$$
Hence, we infer that  $\eta(w)\leq C_1$ for some constant (independent of $v)$, say $C_1=C_1(n,q, C,|U|)>0$ ,  so that the ball of radius $C_1$ in $\mathcal{X}^{1,2}_D(U)$ is invariant under $T$ in $\mathcal{X}^{1,2}_{\mathcal{D}}(U)$.  We need to make sure that the operator $T$ from $\mathcal{X}^{1,2}_{\mathcal{D}}(U)$ to $\mathcal{X}^{1,2}_{\mathcal{D}}(U)$ is continuous and compact to apply the Schauder's Fixed Point Theorem over $T$ to obtain a solution to \eqref{P_nq} in $\mathcal{X}^{1,2}_{\mathcal{D}}(U)$.\\
\textbf{Claim(1) $T$ is continuous:} To accomplish this, we shall prove that for $w_k=T\left(v_k\right)$ and $w=T(v)$, it must hold
$$
\lim _{k \rightarrow \infty}\eta(w_k-w)=0 \text { whenever } \lim _{k \rightarrow \infty}\eta(v_k-v)=0.
$$
Observe that based on $v_k \to v$  in $\mathcal{X}^{1,2}_{\mathcal{D}}(U)$, $U$ is bounded and Remark \ref{r2.5}, we have
\begin{equation}\label{scgt}
\begin{aligned}
 v_k \rightarrow v \text { in } L^{r}_{loc}(\R^N), ~1\leq r< 2^* ~~\text{then}~
 v_k \rightarrow v \text { a.e. in } \R^N.
\end{aligned}
    \end{equation}
We know that 
\[\mathcal{L}(w_k-w)= \frac{1}{\left(v_k^{+}+\frac{1}{n}\right)^q}-\frac{1}{\left(v^{+}+\frac{1}{n}\right)^q}~\text{in}~\Omega.\]
So testing it with $(w_k-w)$ and using  H\"{o}lder inequality,  we obtain
\begin{equation}\label{eqin}
\begin{aligned}
 \eta(w_k-w)^2 
 \leq\left\|w_k-w\right\|_{L^{2^*}(U)}\left(\int_{U}\left(\frac{1}{\left(v_k^{+}+\frac{1}{n}\right)^q}-\frac{1}{\left(v^{+}+\frac{1}{n}\right)^q}\right)^{\frac{2 N}{N+2 }} d x\right)^{\frac{N+2}{2N }},
\end{aligned}
    \end{equation}
where $\frac{2 N}{N+2 }<2^*$. From Remark \ref{r2.5}, for some constant $C(U)>0$ we have
$$
\eta(w_k-w) \leq C(U)\left(\int_{U}\left(\frac{1}{\left(v_k^{+}+\frac{1}{n}\right)^q}-\frac{1}{\left(v^{+}+\frac{1}{n}\right)^q}\right)^{2^*} d x\right)^{1 / 2^*}.
$$
Now, it is easy to see that
$
\frac{1}{\left(v_k^{+}+\frac{1}{n}\right)^q} \leq n^{q}$  and  $\frac{1}{\left(v^{+}+\frac{1}{n}\right)^q} \leq n^{q}
$, so  using  Dominated Convergence Theorem along with \eqref{scgt}, we deduce that
$$
\eta(w_k-w) \rightarrow 0 \text { as } k \rightarrow \infty.
$$
This proves that $T$ is continuous from $\mathcal{X}^{1,2}_{\mathcal{D}}(U)$ to $\mathcal{X}^{1,2}_{\mathcal{D}}(U)$.\\
\textbf{Claim(2) $T$ is compact:} Let us take a bounded sequence $\left\{v_k\right\}_{k \geq 1} \in\mathcal{X}^{1,2}_{\mathcal{D}}(U)$   i.e.  there exists a constant $M>0$ such that $\eta(v_k) \leq M$, for all $k$. Therefore, by compact embedding (see Remark \ref{r2.5}) we deduce that, up to a subsequence,
\begin{align}
\begin{cases}
v_k \rightharpoonup v \mbox{ weakly in } \mathcal{X}^{1,2}_{{\mathcal{D}}}(U)\\
v_k \to v \mbox{ strongly in } \ L^r _{\mathrm{loc}}(\R^N),~~1\leq r<2^*,\\ 
v_k \to v \mbox{ a.e in }  \R^N.
\end{cases}
\end{align}
Furthermore, for some $M_1>0$ which is independent of $k$, we have
$
\eta(T(v_k)) \leq M_1,
$
since $T$ is continuous. Thus, for some $\bar{w}\in \mathcal{X}^{1,2}_D(U)$ we can write
\begin{equation}
\begin{aligned}\label{12405}
 T\left(v_k\right) \rightharpoonup \bar{w} \text { in } \mathcal{X}^{1,2}_{\mathcal{D}}(U),~ T\left(v_k\right) \rightarrow \bar{w} \text { in } L^r_{loc}(\R^N),~\text{for}~ 1 \leq r<2^*.
\end{aligned}
    \end{equation}
Due to the continuity of  $T$, necessarily $\bar{w}=T(v)$. A similar approach as in \eqref{eqin} gives
$$\eta(T(v_k)- T(v))^2\leq \|T(v_k)-T(v)\|_{L^{2}(U)}\left(\int_{U}\left|\frac{1}{\left(v_k^{+}+\frac{1}{n}\right)^q}-\frac{1}{\left(v^{+}+\frac{1}{n}\right)^q}\right|^{2} d x\right)^{\frac{1}{2 }},$$
so using \eqref{12405}, we conclude that 
$$
\lim _{k \rightarrow \infty} \eta\left(T(v_k)-T(v)\right)=0.
$$
Hence, $T$ is compact from $\mathcal{X}^{1,2}_{\mathcal{D}}(U)$ to $\mathcal{X}^{1,2}_{\mathcal{D}}(U)$.
So,  Schauder Fixed Point Theorem provides the existence of $u_n \in\mathcal{X}^{1,2}_{\mathcal{D}}(U)$, for each $n$ such that $u_n=T\left(u_n\right)$, i.e. $u_n$ solves \eqref{eq401}. 
Next, easy to see that $\frac{1}{u_n+ \frac{1}{n}}\in L^{\infty}(U)$ then by using Lemma \ref{linfty}, we have $u_n\in L^{\infty}(U). $ 
A similar approach as in the proof of Lemma \ref{l3.1} establishes the uniqueness of the solution.
  \end{proof}
  
  Furthermore, we establish the following monotonicity property for the sequence 
 $\{u_n\}$ which are the solutions of \eqref{P_nq}.
\begin{lemma}\label{u_n incre}
The sequence $\{u_n\} \subset \mathcal{X}^{1,2}_{\mathcal{D}}(U)\cap L^{\infty}(U)$ as obtained in Lemma \ref{lem4.1}, satisfies
\begin{enumerate}
    \item $u_n\leq u_{n+1},$ a.e. in $\R^N.$
    \item Furthermore, for some constant $\beta=\beta(\omega)>0,$ 
\begin{equation}\label{increseq}
u_n(x) \geq \beta >0,~\text {for a.e.}~ x \in  \omega \Subset \Omega,  ~\forall~ n \in \mathbb{N}.
  \end{equation}
\end{enumerate}
    \end{lemma}
\begin{proof}

Subtracting the problems satisfied by $u_n$ and $u_{n+1}$, the following holds in $\Omega$
\begin{equation}\label{eq309}
\begin{aligned}
\mathcal{L}\left(u_n-u_{n+1}\right) & =\frac{1}{\left(u_n+\frac{1}{n}\right)^q}-\frac{1}{\left(u_{n+1}+\frac{1}{n+1}\right)^q} \leq \frac{1}{\left(u_n+\frac{1}{n+1}\right)^q}-\frac{1}{\left(u_{n+1}+\frac{1}{n+1}\right)^q},\\
& = \frac{\left(u_{n+1}+\frac{1}{n+1}\right)^q-\left(u_n+\frac{1}{n+1}\right)^q}{\left(u_n+\frac{1}{n+1}\right)^q\left(u_{n+1}+\frac{1}{n+1}\right)^q}.
\end{aligned}
    \end{equation}
By multiplying both sides of \eqref{eq309} with $\left(u_n-u_{n+1}\right)^{+}$ and integrating both sides, we get
\begin{equation}\label{eq3012}
\int_\Omega\mathcal{L}(u_n-u_{n+1}) \left(u_n-u_{n+1}\right)^{+}=\int_\Omega \frac{\left(u_{n+1}+\frac{1}{n+1}\right)^q-\left(u_n+\frac{1}{n+1}\right)^q}{\left(u_n+\frac{1}{n+1}\right)^q\left(u_{n+1}+\frac{1}{n+1}\right)^q}\left(u_n-u_{n+1}\right)^{+}\leq 0.
    \end{equation}
We set $w_n= u_n-u_{n+1}$
and aim to show that the following inequality holds
\begin{equation}\label{eq312}
    \int_{\Omega} \mathcal{L} (w_n) w_n^+\,dx \geq  \int_{\Omega} \mathcal{L}  (w_n^+) w_n^+\,dx.
\end{equation}
On the left hand side of \eqref{eq312}, we see that
\begin{equation}\label{eq313}
    \int_{\Omega} \mathcal{L}( w_n) w_n^+\,dx =  \int_{\Omega} \mathcal{L}  (w_n^+) w_n^+\,dx- \int_{\Omega} \mathcal{L} ( w_n^-) w_n^+\,dx.
\end{equation}
Using straightforward calculations, $\left(w_n^-(x)-w_n^-(y)\right)\left(w_n^+(x)-w_n^+(y)\right)\leq 0$ a.e. {in $\mathbb R^N$}, it follows that the second integral on the right-hand side of  \eqref{eq313} satisfies
$\int_{\Omega} \mathcal{L} ( w_n^-) w_n^+\,dx\leq 0.$
Thus, using  this fact in \eqref{eq313}, we obtain  \eqref{eq312}.
Hence,  we deduce that
$
0 \leq \eta(w_n^+)\leq 0.
$ 
Therefore $u_n \leq u_{n+1}$ a.e. in $\Omega$. This completes the proof of the first assertion.
Concerning the second assertion, applying Lemma \eqref{l3.1} for every $\omega \Subset \Omega$, there exists a constant $\beta=\beta(\omega)>0$ satisfying 
$$
u_1(x) \geq \beta >0,~\text {for a.e. } x \in  \omega \Subset \Omega. 
$$
Using the monotonicity argument, we can easily see that  $u_n\geq u_1$ a.e. in $\Omega$, $\forall~ n\in \N$. Thus, for every $n\in\N$, we have $u_n(x)\geq \beta(\omega)>0$ for a.e.  $x \in  \omega \Subset \Omega$, where $\beta(\omega)>0$ is a positive constant that is independent of $n$. Hence, the proof is complete now.
\end{proof}

  \subsection{Existence of weak solution}
 The objective now is to pass to the limit in the sequence \(\{u_n\}\) to obtain a solution for problem \eqref{P_q}.  We first establish some boundedness results for the sequence of positive (over any compact subset of $\Omega$) solutions \(\{u_n\}\) as stated in Lemma \ref{u_n incre} in two different cases below. These results play a crucial role in deriving the existence and regularity results for the problem \eqref{P_q}.

 \subsubsection{\texorpdfstring{Case: $q \in (0,1]$}{}}

We consider the low range of $q$ in this subsection.
\begin{lemma}\label{lem4.3}
Let \(\{u_n\} \subset \mathcal{X}^{1,2}_{\mathcal{D}}(U) \cap L^{\infty}(U)\) be a sequence of solutions to the approximated problem \eqref{P_nq}, as established in Lemma \ref{lem4.1}. Suppose \( 0 < q \leq 1 \) then
 the sequence \(\{u_n\}\) is uniformly bounded in \(\mathcal{X}^{1,2}_{\mathcal{D}}(U)\).
\end{lemma}
\begin{proof}
\textbf{Case I:} If $q=1$, taking  $u_n$ as a test function in \eqref{P_nq},  we obtain that for $q=1$
\begin{align*}
\eta(u_n)^2 \leq \int_{\Omega}  \frac{u_n}{u_n+1/n}dx \leq \int_{\Omega} 1~dx<+\infty.
\end{align*}
\text{Thus}, $\eta(u_n)\leq \kappa$, where $\kappa>0$ which is independent of $n.$ 

\textbf{Case II:} If $0<q<1$, taking again $u_n$ as a test function in $\eqref{P_nq}$, we have
\begin{equation}
    \begin{aligned}
\int_{\Omega} |\nabla u_n(x)|^2\,dx+\int_{Q }\frac{|u_n(x)-u_n(y)|^2}{|x-y|^{N+2 s}} d x d y &=\int_{\Omega} \frac{ u_n}{\left(u_n+\frac{1}{n}\right)^q}\,dx\leq \int_{\Omega} u_n^{1-q}\, dx.
       \end{aligned}
\end{equation}
Then using H\"{o}lder inequality and from Remark \ref{r2.5}, we obtain that 
$$
\begin{aligned}
\eta(u_n)^2 \leq \int_{\Omega}  u_n^{1-q}\,dx 
&\leq\left(\int_{\Omega} 1\,dx\right)^{\frac{1}{m}} \left(\int_{\Omega} u_n^{ {(1-q)}m'}\right)^{\frac{1}{m'}}  \leq C |\Omega|^{\frac{1}{m}}\eta(u_n)^{\frac{2^*}{m'}},
 \end{aligned}
$$
 where  $m'=\frac{m}{m-1}$  
and  $\frac{2^*}{m^\prime}=(1-q)<2$ and $C$ is embedding constant.
Thus, $\eta(u_n)\leq \kappa'$, $\kappa'>0$ which is independent of $n.$
Hence,  sequence $\{u_n\}$ is uniformly bounded in $\mathcal{X}^{1,2}_{\mathcal{D}}(U)$ for $0<q\leq 1$ which finishes the proof.
\end{proof}

\begin{theorem}\label{existence}
 Suppose $0<q\leq   1$. Then there exists a  solution $u \in\mathcal{X}^{1,2}_{\mathcal{D}}(U)$ of problem \eqref{P_q}.    \end{theorem}
\begin{proof}
From the previous Lemma \ref{lem4.3}, we have seen $\{u_n\}$ is uniformly bounded in the Hilbert space $\mathcal{X}^{1,2}_{\mathcal{D}}(U)$. Then, up to a subsequence, there exists $u$ such that $u_n\rightharpoonup u$ weakly convergent in $\mathcal{X}^{1,2}_{\mathcal{D}}(U)$ and $u_n\rightarrow u$ in $L^r_{loc}(\R^N)$ for $1\leq r< 2^*$ and $u_n\rightarrow u$ a.e. in $\R^N$. Moreover since $\{u_n\}$ is an increasing sequence, so $u \geq u_n$ a.e. in $\mathbb R^N$ for each $n$ and
\eqref{increseq}  satisfies for $u$ too. Therefore, by the weak convergence, we can  pass  the limit as $n\to \infty$  to get
$$
\begin{aligned}
 &\lim _{n \rightarrow \infty}\left(\int_{\Omega}\nabla u_n(x).\nabla \varphi(x)\,dx +\int_{Q} \frac{\left(u_n(x)-u_n(y)\right)(\varphi(x)-\varphi(y))}{|x-y|^{N+2 s}} d x d y\right)\\
 &\qquad\qquad\qquad\qquad=\int_{\Omega}\nabla u(x).\nabla \varphi(x)\,dx
\quad+\int_{Q} \frac{\left(u(x)-u(y)\right)(\varphi(x)-\varphi(y))}{|x-y|^{N+2 s}} d x d y, {~\forall~\varphi\in C^{\infty}_0(\Omega)}. 
   \end{aligned}
  $$
 Furthermore, by Lemma \eqref{u_n incre}, we have for all $\omega \Subset \Omega$, there exists a $\beta>0$ such that $u(x)\geq \beta>0,$ for a.e. $x\in \omega.$
Then, on the right hand side of \eqref{P_nq}, we observe that for every $\forall~\varphi\in C^{\infty}_0(\Omega)$ with $\operatorname{supp}(\varphi)=\omega \Subset  \Omega$,  we have
$$
0 \leq\left|\frac{ \varphi}{\left(u_n+\frac{1}{n}\right)^q}\right| \leq \frac{||\varphi||_{L^{\infty}(\Omega)}}{\beta^q} \in L^1(\Omega) .
$$
Hence, using the Lebesgue Dominated Convergence Theorem, we have
$$
\lim _{n \rightarrow \infty} \int_{\Omega} \frac{ \varphi}{\left(u_n+\frac{1}{n}\right)^q}=\int_{\Omega} \frac{ \varphi}{u^q}.
$$
This proves the existence of the solution to \eqref{P_q}.
    \end{proof}

\subsubsection{\texorpdfstring{Case: $q \in (1,\infty)$}{}}
   
We deal with the highly singular case of the problem $(P_q)$ in this subsection. We first recall the following algebraic inequality, as stated in [\cite{Abdellaoui}, Lemma 2.22].
\begin{lemma}\label{algineq}
\begin{enumerate}[(1)]
    \item Let $\alpha > 0$ then for every $x, y \geq 0$ one has
    \[
    (x - y)(x^\alpha - y^\alpha) \geq \frac{4\alpha}{(\alpha + 1)^2} \left( x^{\frac{\alpha + 1}{2}} - y^{\frac{\alpha + 1}{2}} \right)^2.
    \]
\item Let $\alpha \geq 1$ then there exists a constant $C_\alpha>0$ depending only on $\alpha$ such that
    \[
    \left| x + y \right|^{\alpha - 1} \left| x - y \right| \leq C_\alpha \left| x^\alpha - y^\alpha \right|.
    \]
\end{enumerate}
\end{lemma}
\begin{lemma}\label{lem3.5}
Suppose $q>1$ then  $\left\{ u_n^{\frac{q+1}{2}} \right\}$ and  $\{u_n\}$ are uniformly bounded in $\mathcal{X}^{1,2}_{\mathcal{D}}(U)$ and $\mathcal{X}^{1,2}_{\mathcal{D},{loc}}(U)$ respectively, where $\{u_n\}$ denotes solution of \eqref{P_nq},.
    \end{lemma}
\begin{proof}
From Lemma \ref{lem4.1}, we have $u_n\in L^{\infty}(U)$. Utilizing [\cite{Sharma}, Proposition 2.7] and thus choosing $u_n^q\in \mathcal{X}^{1,2}_{\mathcal{D}}(U) $ as a test function in \eqref{P_nq}, we have
\begin{equation}\label{eq316}
\int_{\Omega} \nabla u_n(x)\cdot \nabla u_n^q\,dx+\int_{Q }\frac{\left(u_n(x)-u_n(y)\right)\left(u_n^q(x)-u_n^q(y)\right)}{|x-y|^{N+2 s}} d x d y =\int_{\Omega} \frac{ u_n^q}{\left(u_n+\frac{1}{n}\right)^q}\,dx.
    \end{equation}
We observe that
\begin{equation}\label{eq317}
\int_{\Omega} \nabla u_n\cdot \nabla u_n^q\,dx=\int_{\Omega}\frac{4q }{(q+1)^2}\left|\nabla u_n^{\frac{q+1}{2}}\right|^2\,dx.    
\end{equation}
Using the assertion $(1)$ of Lemma \ref{algineq}, we can see
\begin{equation}\label{eq318}
\int_{Q }\frac{\left(u_n(x)-u_n(y)\right)\left(u_n^q(x)-u_n^q(y)\right)}{|x-y|^{N+2 s}} d x d y \geq \frac{4q }{(q+1)^2} 
\int_{Q }\frac{\left(u_n^{\frac{q+1}{2}}(x)-u_n^{\frac{q+1}{2}}(y)\right)^2}{|x-y|^{N+2 s}} d x d y.   
\end{equation}
Combining \eqref{eq316},\eqref{eq317}, and \eqref{eq318}, we obtain that
$$
\begin{aligned}
    \frac{4q }{(q+1)^2} \left[\int_{\Omega}\left|\nabla u_n^{\frac{q+1}{2}}\right|^2\,dx+ 
\int_{Q }\frac{\left(u_n^{\frac{q+1}{2}}(x)-u_n^{\frac{q+1}{2}}(y)\right)^2}{|x-y|^{N+2 s}} d x d y\right]&\leq \int_{\Omega} \frac{ u_n^q}{\left(u_n+\frac{1}{n}\right)^q}\,dx\leq \int_{\Omega} 1 \,dx,
\end{aligned}
$$ that is 
\begin{equation}\label{eq3019}
\int_{\Omega}\left|\nabla u_n^{\frac{q+1}{2}}\right|^2\,dx+ 
\int_{Q }\frac{\left(u_n^{\frac{q+1}{2}}(x)-u_n^{\frac{q+1}{2}}(y)\right)^2}{|x-y|^{N+2 s}} d x d y\leq \frac{(q+1)^2}{4q} |\Omega|= C,    
\end{equation}
 where $C=C(q, |\Omega|)$ is a positive constant.
Thus,  $\eta\left(u_n^{\frac{q+1}{2}}\right)\leq C$ which implies that the sequence $\left\{u_n^{\frac{q+1}{2}}\right\}$ is uniformly bounded in $\mathcal{X}^{1,2}_{\mathcal{D}}(U).$
Moreover, since $q>1$ and $U$ is bounded, $\{u_n\}$ is uniformly bounded in $L^{q+1}_{loc}(\R^N)$. In particular, $\{u_n\}$ is uniformly bounded in $L^{q+1}(\omega)$, for every $\omega\Subset \Omega$. 

Now, we aim to show that sequence $\{u_n\}$ is uniformly bounded in $\mathcal{X}^{1,2}_{\mathcal{D},{loc}}(U).$ Suppose $\omega \Subset \Omega $. Thanks to second assertion of Lemma \ref{u_n incre} then we have  $u_n\geq \beta>0$ in $\omega$, for each $n\in\N,$ where $\beta$ depends on $\omega$ and independent of $n$. So using \eqref{eq3019}, we have
\begin{equation}\label{eq320}
\int_{\omega} |\nabla u_n|^2\,dx= \frac{4}{(q+1)^2}\int_{\omega} u_n^{1-q} \left|\nabla u_n^{\frac{q+1}{2}}\right|^2\,dx\leq  \frac{4}{(q+1)^2} \beta ^{q-1} C_1, 
    \end{equation}
where $C_1(q, \omega)>0$ is a constant. Furthermore, as $K'=\omega \times \omega \subset \Omega \times \Omega \subset Q$  then from \eqref{eq318} and \eqref{eq3019}, we may deduce that
\begin{equation}\label{eq3k}
\int_{K'} \frac{(u_n(x) - u_n(y))(u_n^q(x) - u_n^q(y))}{|x - y|^{N + 2s}} \, dx \, dy \leq C_2
    \end{equation}
 for some $C_2>0$. We now apply the second assertion of Lemma \ref{algineq}, to get
\begin{equation}\label{eq3022}
 \int_{K'} \frac{|u_n(x) - u_n(y)|^2 \left| u_n(x) + u_n(y) \right|^{q - 1}}{|x - y|^{N + 2s}} \, dx \, dy \leq C_2.
    \end{equation}
As we know that $\{u_n\}\subset L^\infty(U)$ is uniformly bounded, from \eqref{eq320}, \eqref{eq3022} one now gets
\[
 \int_{K'} \frac{|u_n(x) - u_n(y)|^2}{|x - y|^{N + 2s}} \, dx \, dy \leq \frac{M^{q - 1} C_2}{\beta^{1- q}} \quad \text{and} \quad \int_\omega |\nabla u_n|^2 \, dx  \leq \frac{C_3}{{\beta}^{1 - q}}, ~\text{where}~ M, C_2, C_3>0. \tag{4.6}
\]
Thus, we  conclude that $\{u_n\}$ is uniformly bounded in $\mathcal{X}^{1,2}_{\mathcal{D},{loc}}(U).$ 
\end{proof}
We now focus on the existence result for $q>1.$
\begin{theorem}\label{thm3.8}
Suppose $q>1$ then the problem \eqref{P_q} admits a weak solution $u\in \mathcal{X}^{1,2}_{\mathcal{D},{loc}}(U)$ such that $u^{\frac{q+1}{2}} \in \mathcal{X}^{1,2}_{\mathcal{D}}(U).$
\end{theorem}
\begin{proof}
    From Lemma \eqref{lem3.5}(see \eqref{eq3019}), we know that the sequence  
$
\left\{ u_n^{\frac{q + 1}{2}}\right\}
$
is uniformly bounded in Hilbert space $\mathcal{X}^{1,2}_{\mathcal{D}}(U)$, that is for some $C>0$
\begin{equation}
 \operatorname{sup}_{n\in \N}\left[\eta\left( u_n^{\frac{q + 1}{2}}\right)\right]\leq C.
     \end{equation}
By the weak convergence definition, it is easy to see that
     $u^{\frac{q+1}{2}} \in \mathcal{X}^{1,2}_{\mathcal{D}}(U)$ and also $u\in L^2_{loc}(\R^N)$, since $q>1.$
Since by Lemma \ref{lem3.5}, $\{u_n\}$ is uniformly bounded in $\mathcal{X}^{1,2}_{\mathcal{D},{loc}}(U)$ and an increasing sequence for each $n\in \N$, it admits a weak limit  $u>0$ a.e. in $\Omega$ (over any compact subset of $\Omega$) in $\mathcal{X}^{1,2}_{\mathcal{D},{loc}}(U)$ as $n\to \infty$, see Lemma \ref{u_n incre}, which shall be pointwise limit as well. 
Hence,  following the proof of Theorem \ref{existence}, for any  $\varphi \in C^{\infty}_0(\Omega)$, we obtain
\[
\lim_{n \to \infty} \int_{\Omega} \nabla u_n \nabla \varphi \, dx = \int_{\Omega}  \nabla u \nabla \varphi \, dx  
\]
and for the subsequent term, we obtain the estimate
\begin{equation}\label{eq321}
\begin{aligned}
&\left| \int_{Q}   \frac{\left(u_n(x)- u_n(y)\right) (\varphi(x) - \varphi(y))}{|x-y|^{N+2s}} \, \,dx\,dy - \int_{Q}  \frac{(u(x)-u(y)) (\varphi(x) - \varphi(y))} {|x-y|^{N+2s}} \, \,dx\,dy\right|\\
\quad& \leq  \int_{Q}   \frac{\left|\left(u_n(x)- u(x)\right) - \left(u_n(y)- u(y)\right)\right| \left|\varphi(x) - \varphi(y)\right|}{|x-y|^{N+2s}} \, \,dx\,dy. 
    \end{aligned}
    \end{equation}
    Next, let us fix 
$\epsilon>0$. We claim that there exists a compact set 
$\mathcal{K}\subset Q$ such that
\begin{equation}\label{eq322}
    \int_{Q\setminus \mathcal{K}}   \frac{\left|\left(u_n(x)- u(x)\right) - \left(u_n(y)- u(y)\right)\right| \left|\varphi(x) - \varphi(y)\right|}{|x-y|^{N+2s}} \, \,dx\,dy\leq \frac{\epsilon}{2}.
\end{equation}
To establish \eqref{eq322}, we follow the arguments of  [\cite{Montoro}, Theorem 3.6].  On the other hand, for any arbitrary measurable subset $\mathcal{E}\subset\mathcal{K}$, we obtain the inequality
$$\int_{ \mathcal{E}}   \frac{\left|\left(u_n(x)- u(x)\right) - \left(u_n(y)- u(y)\right)\right| \left|\varphi(x) - \varphi(y)\right|}{|x-y|^{N+2s}} \, \,dx\,dy \leq C \left(\int_{\mathcal{E}} \frac{|\varphi(x)-\varphi(y)|^2}{|x-y|^{N+2s}}\,dx\,dy\right)^{\frac{1}{2}}.$$
If the measure $|\mathcal{E}|\to 0$ then 
$$\int_{ \mathcal{E}}   \frac{\left|\left(u_n(x)- u(x)\right) - \left(u_n(y)- u(y)\right)\right| \left|\varphi(x) - \varphi(y)\right|}{|x-y|^{N+2s}} \, \,dx\,dy \rightarrow 0,~~~\text{uniformly on $n$}.$$
Moreover, 
$$\frac{\left|\left(u_n(x)- u(x)\right) - \left(u_n(y)- u(y)\right)\right| \left|\varphi(x) - \varphi(y)\right|}{|x-y|^{N+2s}} \to 0~\text{a.e. in $Q$}.$$
 Now, using the Vitali convergence Theorem, for given $\epsilon>0,$ there exists $N_0>0$ such that, if $n\geq N_0$ then
\begin{equation}\label{eq323}
     \int_{ \mathcal{K}}   \frac{\left|\left(u_n(x)- u(x)\right) - \left(u_n(y)- u(y)\right)\right| \left|\varphi(x) - \varphi(y)\right|}{|x-y|^{N+2s}} \, \,dx\,dy\leq \frac{\epsilon}{2}
\end{equation}
Thus, taking the limit as 
$n\to \infty$ on both sides of \eqref{eq321} and applying \eqref{eq322} and \eqref{eq323}, we obtain
$$\lim_{n\to \infty} \int_{{Q}}   \frac{\left(u_n(x)- u_n(y)\right) (\varphi(x) - \varphi(y))}{|x-y|^{N+2s}} \, \,dx\,dy = \int_{{Q}}  \frac{(u(x)-u(y)) (\varphi(x) - \varphi(y))} {|x-y|^{N+2s}} \, \,dx\,dy,~~\forall~~\varphi\in {C^{\infty}_0(\Omega)}.$$
 Also, by Lemma \eqref{u_n incre}, we have for all $\omega \subset \subset \Omega$ there exists $\beta>0$ such that $u(x)\geq \beta>0,$ for $x\in \omega.$
Finally following the same arguments as in Theorem \ref{existence} for handling R.H.S. of \eqref{P_nq}, we establish the existence of the weak solution to \eqref{P_q}.
\end{proof}

\subsection{Regularity and Uniqueness}
 Initially, we will establish the following result, which enables us to choose a test function from the space $\mathcal{X}^{1,2}_{\mathcal{D}}(U)$ 
in  \eqref{dd2.2}.
\begin{Proposition}\label{prop3.1}
   Let $q>0$ and $u \in \mathcal{X}^{1,2}_{\mathcal{D}}(U)$ be a weak solution to the problem \eqref{P_q}, then \eqref{dd2.2} satisfies for every $\varphi \in \mathcal{X}^{1,2}_{\mathcal{D}}(U).$ 
\end{Proposition} 
\begin{proof}
Suppose, \( u \in \mathcal{X}^{1,2}_{\mathcal{D}}(U) \) is a solution to the problem \eqref{P_q}. For any test function \( \varphi \in C^{\infty}_0(\Omega) \), we have

\begin{equation}\label{eq324}
\int_{\Omega}\nabla u \cdot \nabla \varphi \, dx + \int_{{Q}}  \frac{(u(x)-u(y)) (\varphi(x) - \varphi(y))} {|x-y|^{N+2s}} \, dx \, dy = \int_{\Omega} \frac{\varphi(x)}{u^q(x)}  \, dx.
\end{equation}
Furthermore, using the density argument, for every \( \phi \in \mathcal{X}^{1,2}_{\mathcal{D}}(U) \), there exists a sequence of functions \( 0 \leq \phi_n \in C^{\infty}_0(U) \) such that \( \phi_n \rightarrow |\phi| \) strongly in \( \mathcal{X}^{1,2}_{\mathcal{D}}(U) \) as \( n \rightarrow \infty \) that is $\eta(\phi_n) \to \eta(|\phi|)$ and also pointwise a.e. in \( \mathbb{R}^n \). Using the definition of lower semicontinuity and the Cauchy-Schwarz inequality, we obtain that
\begin{equation}\label{eq325}
\begin{aligned}
&\left|\int_{\Omega} \frac{1}{u^q(x)} \phi \, dx\right| \leq \int_{\Omega} \frac{1}{u^q(x)} |\phi| \, dx \leq \liminf _{n \rightarrow \infty} \int_{\Omega} \frac{1}{u^q(x)} \phi_n \, dx \\
&= \liminf _{n \rightarrow \infty} \left(\int_{\Omega} \nabla u \cdot \nabla \phi_n \, dx + \int_{Q} \frac{(u(x)-u(y)) (\phi_n(x) - \phi_n(y))} {|x-y|^{N+2s}} \, dx \, dy \right) \\
&\leq C \liminf _{n \rightarrow \infty}  \left( \eta(u)\eta(\phi_n)  \right) \leq C \eta(u) \eta(|\phi|)\leq C \eta(u) \eta(\phi),
\end{aligned}
\end{equation}
where \( C > 0 \) is a constant  independent of \( n \). Now, for \( \varphi \in \mathcal{X}^{1,2}_{\mathcal{D}}(U) \), there exists a sequence \( \varphi_n \in C^{\infty}_0(U) \) that converges to \( \varphi \) strongly in \( \mathcal{X}^{1,2}_{\mathcal{D}}(U) \). Proceeding similarly by taking \( \phi = \varphi_n - \varphi \) in \eqref{eq325}
and passing limit as $n\to \infty$ on both sides, we get
\begin{equation}\label{eq326}
\lim_{n\to \infty} \left|\int_{\Omega} \frac{1}{u^q(x)} \left(\varphi_n - \varphi\right) \, dx\right|\leq C \eta(u) \lim_{n\to \infty} \eta(|\varphi_n-\varphi|)\leq C \eta(u) \lim_{n\to \infty} \eta(\varphi_n-\varphi) =0.
\end{equation}
Again using the fact that \( \varphi_n \rightarrow \varphi \) strongly in \( \mathcal{X}^{1,2}_{\mathcal{D}}(U) \), we obtain
\begin{equation}\label{eq327}
\lim _{n \rightarrow \infty} \left\{ \int_{\Omega} \nabla u \cdot \nabla\left(\varphi_n - \varphi\right) \, dx + \int_{Q} \frac{(u(x)-u(y)) \left(\left(\varphi_n - \varphi\right)(x)-\left(\varphi_n - \varphi\right)(y)\right)} {|x-y|^{N+2s}} \, dx \, dy \right\} = 0.
\end{equation}
Thus, utilizing \eqref{eq327} and \eqref{eq326} in \eqref{eq324} yields the desired result.
\end{proof}
As a consequence of Proposition \ref{prop3.1}, we provide a straightforward proof of the uniqueness result.
\begin{lemma}
 Let $q>0$  then, the problem \eqref{P_q} admits unique weak solution in $\mathcal{X}^{1,2}_{\mathcal{D}}(U)$.
 \end{lemma}
\begin{proof}
 By contradiction, suppose $u_1, u_2 \in \mathcal{X}^{1,2}_{\mathcal{D}}(U)$ are two weak solutions of the problem \eqref{P_q}. Then, by Proposition \ref{prop3.1}, for every $\varphi \in \mathcal{X}^{1,2}_{\mathcal{D}}(U)$, we have
\begin{equation}\label{eq328}
 \int_{\Omega} \nabla u_1 \cdot \nabla \varphi~ d x+\int_{Q}  \frac{(u_1(x)-u_1(y)) (\varphi(x) - \varphi(y))} {|x-y|^{N+2s}} \, \,dx\,dy=\int_{\Omega} \frac{1}{u^q_1(x)} \varphi(x) d x 
     \end{equation}
\begin{equation}\label{eq329}
\int_{\Omega} \nabla u_2 \cdot \nabla \varphi~ d x+\int_{Q}  \frac{(u_2(x)-u_2(y)) (\varphi(x) - \varphi(y))} {|x-y|^{N+2s}} \, \,dx\,dy=\int_{\Omega} \frac{1}{u^q_2(x)} \varphi(x) d x
    \end{equation}
Denoting  $u_1-u_2= w$ and putting $\varphi=\left(u_1-u_2\right)^{-} = w^-\in \mathcal{X}^{1,2}_{\mathcal{D}}(U)$ in \eqref{eq328} and \eqref{eq329} and then subtracting them, we have
\begin{equation}\label{eq3300}
\int_{\Omega} \nabla w\cdot \nabla w^{-} d x 
+\int_{Q} \frac{\left(w(x)-w(y)\right)\left(w^{-}(x)-w^{-}(y)\right)}{|x-y|^{N+2s}} \,dxdy  =\int_{\Omega}\left\{\frac{1}{u^q_1(x)}-\frac{1}{u^q_2(x)}\right\}w^{-}(x) d x \geq 0.
\end{equation}
Through the following easy calculation
$$
\begin{aligned}
   & (w(x)-w(y)) (w^-(x)-w^-(y))
    \leq -\left|\left(w^-(x)-w^-(y)\right)\right|^2
    \end{aligned}
    $$
and $\nabla w \cdot\nabla{w^-} = -|\nabla w^-|^2$ then from \eqref{eq3300} and 
$$0\leq  \int_{\Omega} -|\nabla{w^-}|^2 \,dx -  \int_{Q} {\dfrac{ \left|\left(w^-(x)-w^-(y)\right)\right|^2}{|x-y|^{N+2s}}} ~ dx dy\leq 0$$
which  implies $w^-=0$ a.e. in $\R^N,$ that is $u c_1\geq u_2$  a.e. in $\R^N.$ Similarly, by choosing $\varphi=\left(u_2-u_1\right)^{-} \in \mathcal{X}^{1,2}_{\mathcal{D}}(U)$ in \eqref{eq328} and \eqref{eq329} and repeating above arguments, we can obtain $u_2\geq u_1$ a.e. in $\R^N.$
Hence, the problem \eqref{P_q} admits a unique solution in 
$\mathcal{X}^{1,2}_{\mathcal{D}}(U).$
\end{proof}
\begin{theorem}\label{proof-mt3}
  Let $q>0$. The weak solution provided by Theorem \ref{existence} and Theorem \ref{thm3.8} belongs to $L^{\infty}(U)$. 
\end{theorem}
\begin{proof}
Following the same reasoning as in the proof of Lemma \ref{linfty}, we apply H\"older's inequality and the Stampacchia method to the solution \( u_n \) of \eqref{P_nq}. Consequently, we conclude that \( u_n \in L^\infty(U) \) and there exists a $C>0$(independent of $n$) such that $\|u_n\|_{L^\infty(U)}\leq C$, for all $n\in \mathbb N$. From this, it is straightforward to deduce that \( u \in L^\infty(U) \).
\end{proof}
\begin{definition}
    We say that  for $u,v \in \mathcal{X}^{1,2}_{\mathcal{D}}(U)$, 
$
\mathcal{L}u - u^{-q} \leq \mathcal{L}v - v^{-q}~ \text{weakly in }\Omega
$
if
 $$ \int_{\Omega} \mathcal{L}u \psi\,dx -\int_{\Omega} u^{-q}\psi\,dx\leq \int_{\Omega} \mathcal{L}v\psi\,dx -\int_{\Omega} v^{-q}\psi\,dx,~\forall~ 0\leq \psi \in \mathcal{X}^{1,2}_{\mathcal{D}}(U).$$
\end{definition}
Next, we shall establish the following weak comparison principle.
\begin{lemma}\label{wcp}(Weak Comparison Principle) Suppose $q>0$ and $u, v \in \mathcal{X}^{1,2}_{\mathcal{D}}(U)$ satisfies
$$
\begin{cases}\mathcal{L} v-v^{-q} \geq\mathcal{L} u- u^{-q} & \text {weakly in } \Omega, \\ u=v=0 & \text { in } U^c,\\
\mathcal{N}_s(u)=\mathcal{N}_s(v) =0 &\text{in} ~~{{\mathcal{N}}}\\
\frac{\partial u}{\partial \nu}=\frac{\partial u}{\partial \nu}=0 &\text{in}~~ \partial \Omega \cap \overline{{\mathcal{N}}}.
\end{cases}
$$
Then, $v \geq u$ a.e. in $\mathbb{R}^n$.
    \end{lemma}
\begin{proof}
 Suppose $0\leq \varphi \in\mathcal{X}^{1,2}_{\mathcal{D}}(U)$, we can see
$$
\int_\Omega \mathcal{L} v \varphi~dx-\int_{\Omega} v^{-q} \varphi \,d x \geq\int_\Omega \mathcal{L} u \varphi~dx-\int_{\Omega} u^{-q} \varphi d x .
$$
If we choose $\varphi=(u-v)^{+}$, the above inequality can be written as
\begin{equation}\label{411}
\int_\Omega \mathcal{L}(v-u) (u-v)^+~dx-\int_{\Omega} (u-v)^{+}\left(\frac{1}{v^q}-\frac{1}{u^q}\right) d x \geq 0.
    \end{equation}
But we observe that $\dint_{\Omega} (u-v)^{+}\left(\frac{1}{v^q}-\frac{1}{u^q}\right) d x \geq 0$ which says that
\begin{equation}\label{412}
 \int_\Omega \mathcal{L}(v-u) (u-v)^+~dx\geq 0 .
    \end{equation}
Setting $w=u-v$, we write $w=(u-v)^{+}-(u-v)^{-}$. Also, studying case by case, it is not hard to see that
\begin{equation}\label{413}
[w(y)-w(x)][\varphi(x)-\varphi(y)] \leq-[\varphi(x)-\varphi(y)]^2,
    \end{equation}
and 
\begin{equation}\label{414}
\nabla w(x).\nabla \varphi(x) \leq -[\nabla \varphi(x)]^2~~\text{for almost every }~x,y \in \mathbb{R}^n.
    \end{equation}
Using \eqref{413} and \eqref{414}, we  have
\begin{equation}\label{415}
\begin{aligned}
\int_\Omega \mathcal{L}(v-u) (u-v)^+~dx & =\frac{1}{2}\int_{\mathbb{R}^n} \nabla w(x). \nabla \varphi(x)\,dx- \frac{1}{2}  \int_{\mathbb{R}^{2n}} \frac{[w(y)-w(x)][\varphi(x)-\varphi(y)]}{|x-y|^{N+2 s}} d y d x \\
& \leq \frac{1}{2}\int_{\mathbb{R}^n} -[\nabla \varphi(x)]^2\,dx+\frac{1}{2}  \int_{\mathbb{R}^{2n}} \frac{-[\varphi(x)-\varphi(y)]^2}{|x-y|^{N+2 s}} d y d x \leq 0.
\end{aligned}
    \end{equation}
Combining \eqref{412} with \eqref{415}, we have $\eta(\varphi)=0$. 
 Thus, $v \geq u$ a.e. in $\mathbb{R}^n$.
   \end{proof}
   \begin{remark}\label{rem3.3}
       Let $u_n \in \mathcal{X}^{1,2}_{\mathcal{D}}(U)$ be the solution to problem \eqref{P_nq} as established by Lemma \ref{u_n incre}, and let $\hat{u}$ represent the pointwise limit of $u_n$ in $\mathbb{R}^N$. By Theorem \ref{existence}, $\hat{u} \in \mathcal{X}^{1,2}_{\mathcal{D}}(U)$ is the weak solution of \eqref{P_q}. Furthermore, we have $u_n \leq \hat{u}$ for all $n \in \mathbb{N}$.  
       \end{remark}

\subsection{Complementary properties}
This section contains various additional observations and properties satisfied by the sequence $\{u_n\}$ resulting in exceptional characterizations. 
\begin{Proposition}\label{p3.2}
  \begin{enumerate}
      \item  For any $\psi \in \mathcal{X}^{1,2}_{\mathcal{D}}(U)$ and every $n\in\N$, we have\[ 
\eta(u_n)^2 \leq \eta(\psi)^2 +  \int_{\Omega} \frac{(u_n - \psi) }{\left( u_n + \frac{1}{n} \right)^q} \, dx. 
\]
\item $\{\eta(u_n)\}$ is a non- decreasing sequence and
up to a subsequence $ u_n   \to  \hat{u} $ converges strongly in $\mathcal{X}^{1,2}_{\mathcal{D}}(U)$, as $n\to\infty$, for  $\hat{u}\in \mathcal{X}^{1,2}_{\mathcal{D}}(U)$, given in Remark \ref{rem3.3}.
  \end{enumerate}
\end{Proposition}
\begin{proof} (1)
   By Lemma \ref{lem4.1}, for every \( h \in \mathcal{X}^{1,2}_{\mathcal{D}}(U) \), there exists a unique solution \( v \in \mathcal{X}^{1,2}_{\mathcal{D}}(U) \) to the following problem
 \begin{equation}\label{eq341}
\begin{cases}  
     \mathcal{L}v &= ~\frac{1}{\left( h^+ + \frac{1}{n} \right)^q}, \quad v>0 \quad 
 \text{in }\Omega,\\ 
  v&=~~0~~\text{in} ~~U^c,\\
 \mathcal{N}_s(v)&=~~0 ~~\text{in} ~~{{\mathcal{N}}}, \\
 \frac{\partial v}{\partial \nu}&=~~0 ~~\text{in}~~ \partial \Omega \cap \overline{{\mathcal{N}}},
      \end{cases}
    \end{equation}
which is a minimizer of the functional \( \mathcal{F} : \mathcal{X}^{1,2}_{\mathcal{D}}(U) \to \mathbb{R} \) given by
\[
\mathcal{F}(\psi) = \frac{1}{2} \eta(\psi)^2 - \int_{\Omega} \frac{\psi}{\left( h^+ + \frac{1}{n} \right)^q}  \, dx.
\]
 This means 
\begin{equation}\label{eq336}
\frac{1}{2} \eta(v)^2 - \int_{\Omega} \frac{v}{\left( h^+ + \frac{1}{n} \right)^q}  \, dx \leq \frac{1}{2} \eta(\psi)^2 - \int_{\Omega} \frac{\psi}{\left( h^+ + \frac{1}{n} \right)^q} \, dx~~~\forall~ \psi\in\mathcal{X}^{1,2}_{\mathcal{D}}(U). 
\end{equation}
 Now, choosing $ h = u_n $ in \eqref{eq341} and using Lemma \ref{lem4.1}, from \eqref{eq336} we obtain 
\begin{equation*}
\frac{1}{2} \eta(u_n)^2 - \int_{\Omega} \frac{u_n}{\left( u_n^+ + \frac{1}{n} \right)^q}  \, dx \leq \frac{1}{2} \eta(\psi)^2 - \int_{\Omega} \frac{\psi}{\left( u_n^+ + \frac{1}{n} \right)^q} \, dx, 
\end{equation*}
that is
\begin{equation}\label{eq337}
 \eta(u_n)^2 \leq \eta(\psi)^2 +  \int_{\Omega} \frac{(u_n - \psi) }{\left( u_n^+ + \frac{1}{n} \right)^q} \, dx. 
   \end{equation}
Hence, the proof of the first assertion is complete.

(2) Since, by Lemma \ref{lem4.3}, the sequence \( \{ u_n \} \) is uniformly bounded in \( \mathcal{X}^{1,2}_{\mathcal{D}}(U) \), it follows that, up to a subsequence, \( u_n \rightharpoonup \hat{u} \) weakly in \( \mathcal{X}^{1,2}_{\mathcal{D}}(U) \). Therefore, we have 
\begin{equation}\label{eq339}
\eta(\hat{u}) \leq \lim_{n \to \infty} \eta(u_n).
    \end{equation} For proving second assertion,  taking \( \psi = u_{n+1} \) in  \eqref{eq337} and using the monotonicity property \( u_n \leq u_{n+1} \) from Lemma \eqref{u_n incre}, we get
 \begin{equation*}
 \eta(u_n)^2 \leq \eta(u_{n+1})^2 +  \int_{\Omega} \frac{(u_n - u_{n+1}) }{\left( u_n^+ + \frac{1}{n} \right)^q} \, dx. 
   \end{equation*}
 Thus, we have \( \eta(u_n) \leq \eta(u_{n+1})\), since $u_n-u_{n+1}\leq 0$.
Furthermore, it is easy to observe that \( u_n \leq \hat{u} \), see Remark \ref{rem3.3} and choosing \( \psi = \hat{u} \) in \eqref{eq337}, we obtain
\[
\eta(u_n) \leq \eta(\hat{u}), 
\]
using this and from $\eta(u_n)\leq \eta(u_{n+1})$, for each $n\in\N$, we deduce that
\begin{equation}\label{eq338}
\lim_{n \to \infty} \eta(u_n) \leq \eta(\hat{u}).
    \end{equation}

Hence, from \eqref{eq338}, \eqref{eq339} and using  the uniform convexity of $\mathcal{X}^{1,2}_{\mathcal{D}}(U)$, we get  $ u_n   \to  \hat{u} $  strongly in $\mathcal{X}^{1,2}_{\mathcal{D}}(U)$ as $n\to\infty$. Thus, the proof of the second assertion is complete now.
\end{proof}

\begin{lemma}\label{lem310}
 Let $0<q<1$ and \( \hat{\mathcal{I}} : \mathcal{X}^{1,2}_{\mathcal{D}}(U) \to \mathbb{R} \) be a functional defined as
\[
\hat{\mathcal{I}} (v) = \frac{1}{2} \eta(v)^2 - \frac{1}{1 - q} \int_{\Omega} (v^+)^{1-q}  \, dx.
\]
Then \(\hat{u} \) is a minimizer of the functional \( \hat{\mathcal{I}} \).
   \end{lemma}
\begin{proof}
 Let us  consider the functional \( \mathcal{I}_n : \mathcal{X}^{1,2}_{\mathcal{D}}(U) \to \mathbb{R} \) defined by
\[
\mathcal{I}_n (v) = \frac{1}{2} \eta(v)^2 - \int_{\Omega} \mathcal{G}_n(v)  \, dx,
\]
where
\[
\mathcal{G}_n(v) = \frac{1}{1 - q} \left( v^+ + \frac{1}{n} \right)^{1-q} - n^q v^-.
\]
It is easy to see that \( \mathcal{I}_n \) is bounded below and coercive. As a consequence,  $\mathcal{I}_n$ is \( C^1 \). So, by the weak lower semicontinuity of $\mathcal{I}_n$, functional  \( \mathcal{I}_n \) consists of a minimizer \( v_n \in \mathcal{X}^{1,2}_{\mathcal{D}}(U) \) such that
\[
\langle \mathcal{I}'_n(v_n), \psi \rangle = 0, \quad \text{for all } \psi \in \mathcal{X}^{1,2}_{\mathcal{D}}(U).
\] 
Now we can see \( \mathcal{I}_n(v_n) \leq \mathcal{I}_n(v^+_n) \) from it, we get \( v_n \geq 0 \)  a.e. in  \( \R^N \). Therefore, \( v_n \) solves the approximated problem \eqref{P_nq}. Hence, uniqueness follows from  Lemma \ref{lem4.1}, then  \( u_n = v_n \) a.e. in $\R^n$ and so \( u_n \) is a minimizer of \( \mathcal{I}_n \), since $v_n$ is a minimizer of $\mathcal{I}_n$. Thus,  
\begin{equation}\label{eq340}
\mathcal{I}_n(u_n) \leq \mathcal{I}_n(v^+). 
    \end{equation}
Since \( u_n \leq\hat{u} \), by the Lebesgue Dominated Convergence Theorem, we have
\begin{equation}\label{eq3041}
\lim_{n \to \infty} \int_{\Omega} \mathcal{G}_n(u_n)  \, dx = \frac{1}{1-q} \int_{\Omega} \left(u_n+\frac{1}{n}\right)^{1-q} \, dx= \frac{1}{1-q} \int_{\Omega}\hat{u}^{1-q} \, dx.
    \end{equation}
Moreover, by Proposition \ref{p3.2}, $ u_n   \to  \hat{u} $ converges strongly in $\mathcal{X}^{1,2}_{\mathcal{D}}(U)$, as $n\to\infty$ that is
\begin{equation}\label{eq3042}
\lim_{n \to \infty} \eta(u_n) = \eta(\hat{u}).
    \end{equation}
Hence, from \eqref{eq3041} and \eqref{eq3042}, we obtain
\begin{equation}\label{eq3043}
\lim_{n \to \infty} \mathcal{I}_n(u_n) = \hat{\mathcal{I}}(\hat{u}).
    \end{equation}
Moreover, we see 
\begin{equation}\label{eq342}
 \lim_{n \to \infty} \int_{\Omega} \mathcal{G}_n(v^+)  \, dx = \frac{1}{1 - q} \int_{\Omega} (v^+)^{1-q}  \, dx,~~~ \forall~  v \in \mathcal{X}^{1,2}_{\mathcal{D}}(U) 
   \end{equation}
and observe that \( \eta(v^+) \leq \eta(v) \) and using \eqref{eq3043}, \eqref{eq342} in \eqref{eq340}, we conclude  that \( \hat{\mathcal{I}}(\hat{u}) \leq \hat{\mathcal{I}}(v) \), ~$\forall$~\( v \in \mathcal{X}^{1,2}_{\mathcal{D}}(U) \).
Hence,  \(\hat{u} \) is a minimizer of the functional \( \hat{\mathcal{I}} \).

Thus, the proof of this lemma is now completed.
    \end{proof}
 Let us define
\[\Re(U)= \inf_{v\in \mathcal{X}^{1,2}_{\mathcal{D}}(U)\setminus \{0\}}\left\{\eta(v)^2:\int_\Omega |v|^{1-q}  \, dx = 1 \right\}\]
and 
\[
\mathcal{Z} = \left\{ v \in\mathcal{X}^{1,2}_{\mathcal{D}}(U) : \int_\Omega |v|^{1-q}  \, dx = 1 \right\}.
\]
We will establish the following result, which is crucial for proving the mixed Sobolev inequality over $\mathcal{X}^{1,2}_{\mathcal{D}}(U)$.
    \begin{Proposition}\label{pp3.3}
    Let $0<q<1$. Suppose $\hat{u}\in \mathcal{X}^{1,2}_{\mathcal{D}}(U)$ is a weak solution to the problem \eqref{P_q} given by Theorem \ref{existence}, then 
    \begin{equation}\label{eq30}
   \Re(U)=\left(\eta(\hat{u})\right)^{\frac{-(2+2q)}{1-q}}. 
\end{equation}
\end{Proposition}
\begin{proof}
 We shall establish that
$
\Re(U) = \inf_{v \in \mathcal{Z}} \eta(v)^2 = \eta(\hat{u})^{ \frac{-(2+2q)}{1-q}}.
$
Applying  Proposition \eqref{prop3.1} and taking \(\hat{u}\) as a test function in \eqref{dd2.2} we get 
\begin{equation}\label{eq3045}
\eta(\hat{u})^2 = \int_\Omega \hat{u}^{1-q}  \, dx. 
    \end{equation}
    Let us consider \(\mathcal{V} = \left( \int_\Omega \hat{u}^{1-q}  \, dx \right)^{-\frac{1}{1-q}} \hat{u} \in \mathcal{Z} \), 
then using \eqref{eq3045}, we have
\begin{equation}\label{eq3046}
\eta(\mathcal{V})^2 
=\frac{\eta(\hat{u})^2 }{\left( \dint_\Omega \hat{u}^{1-q}  \, dx \right)^{\frac{2}{1-q}}}= \frac{\eta(\hat{u})^2 }{\left( \eta(\hat{u})\right)^{\frac{4}{1-q}}}= \eta(\hat{u})^{ \frac{-(2+2q)}{1-q}}.
  \end{equation}
Suppose \( v \in \mathcal{Z} \) and \( \lambda = \frac{1}{\eta(v)^{\frac{2}{1 + q }} } \). Using  Lemma \ref{lem310}, we know that \(\hat{u}\) is a minimizer of the functional \( \hat{\mathcal{I}} \), so  \(\hat{\mathcal{I}}  (\hat{u}) \leq \hat{\mathcal{I}}  (\lambda |v|) \). 
We can easily find that
$$
\begin{aligned}
\hat{\mathcal{I}} (\lambda |v|) = \frac{\lambda^2}{2} \eta(|v|)^2 - \frac{\lambda^{1-q}}{1-q} \int_\Omega |v|^{1-q}  \, dx &\leq \frac{\lambda^2}{2} \eta(v)^2 - \frac{\lambda^{1-q}}{1-q}
  =\left( \frac{1}{2} - \frac{1}{1-q} \right) \eta(v)^{\frac{-2+2q}{1+q}},
    \end{aligned}
    $$
    Hence,
    $$\left( \frac{1}{2} - \frac{1}{1-q} \right) \eta(\hat{u})^2 = \hat{\mathcal{I}}  (\hat{u})\leq \hat{\mathcal{I}} (\lambda |v|)=\left( \frac{1}{2} - \frac{1}{1-q} \right) \eta(v)^{\frac{-2+2q}{1+q}}
    $$
   which implies
    $\eta(\hat{u})^2\leq  \eta(v)^{\frac{-2+2q}{1+q}}. $
Since \( v \in \mathcal{Z} \) is arbitrary so it follows that 
{\begin{equation}\label{eq3047}
\eta(\hat{u})^{ \frac{-(2+2q)}{1-q}} \leq \inf_{v \in \mathcal{Z}} \eta(v)^2. 
    \end{equation}
As knowing \(\mathcal{V} \in \mathcal{Z}\) and from \eqref{eq3046} and \eqref{eq3047}, we have
$\eta(\mathcal{V})^2\leq \inf_{v \in \mathcal{Z}} \eta(v)^2$
which finishes the proof.}
    \end{proof}
Next, we aim to establish the mixed Sobolev Inequality, as stated in the following main theorem.
\begin{theorem}
 Let $0<q<1$ then for every $v\in \mathcal{X}^{1,2}_{\mathcal{D}}(U)$
\begin{equation}\label{eq3049}
    \eta(v)^2\geq C\left(\int_{\Omega}|v|^{1-q}\,dx\right)^{\frac{2}{1-q}}
\end{equation}
holds if and only if $C\leq \Re(U).$ Moreover, if \eqref{eq3049} is an equality for some $w\in \mathcal{X}^{1,2}_{\mathcal{D}}(U) \setminus \{0\}$, then $w= k \hat{u}$, for some constant $k>0$.
 \end{theorem}
   \begin{proof}
Let us suppose that  \eqref{eq3049}
holds and on contrary, $C > \Re(U)$. From  \eqref{eq30} and \eqref{eq3046}, it is easy to see that
\[
C \left( \int_\Omega\mathcal{V}^{1-q}  \, dx \right)^{\frac{2}{1-q}} > \eta(\mathcal{V})^2
\]
and get a contradiction with inequality \eqref{eq3049}. Conversely, let us assume that
$
C \leq \Re(U) = \inf_{v \in \mathcal{Z}} \eta(v)^2 \leq \eta(w)^2,~\forall~w \in \mathcal{Z}.
$
The claim \eqref{eq3049} clearly holds if \( v \equiv 0 \) a.e. in \( \mathbb{R}^N \). Thus, we focus on the case of \( v \in \mathcal{X}^{1,2}_{\mathcal{D}}(U) \setminus \{0\} \), for which we have
\[
w = \left( \int_\Omega |v|^{1-q}  \, dx \right)^{-\frac{1}{1-q}} v \in \mathcal{Z}
\]
which straightaway gives
\[
C \leq \left( \int_\Omega |v|^{1-q}  \, dx \right)^{-\frac{2}{1-q}} \eta(v)^2,
\]
completing the first part of the proof.

 Using \eqref{eq3046} in \eqref{eq30}, we have $\Re(U) = \eta(\mathcal{V})^2$. Let $v \in \mathcal{Z}$ be such that $\Re(U) = \eta(v)^2$. Firstly, we claim that $v$ has a constant sign in $\R^N$.   By contrary argument, suppose  $v$ changes sign in $\R^N$. Using the fact that
\[
\left||v(x)| - |v(y)|\right| \leq |v(x) - v(y)|,
\]
it is easy to see that
\begin{equation}\label{eq3050}
\eta( |v|)^2 \leq \eta(v)^2.
    \end{equation}
Since $|v| \in \mathcal{Z}$ as well, we have
$
\eta(v)^2 = \Re(U) \leq \eta( |v|)^2.
$
From this along with \eqref{eq3050}. Hence, $v$ has a constant sign in $\R^N$. Thus,  $v \geq 0$ in $\R^N$. Furthermore, since  $ \mathcal{V}, v \geq 0$ and $0<1-q<1$, we obtain that
$$
\begin{aligned}
\text{(say)}\zeta =\left(\int_{\Omega}\left(\frac{v}{2}+\frac{\mathcal{V}}{2}\right)^{1-q}  d x\right)^{\frac{1}{1-q}}
&\geq\left(\int_{\Omega}\left(\frac{v}{2}\right)^{1-q} d x\right)^{\frac{1}{1-q}}+\left(\int_{\Omega}\left(\frac{\mathcal{V}}{2}\right)^{1-q}  d x\right)^{\frac{1}{1-q}} \\
&=\frac{1}{2}\left[\left(\int_{\Omega} v^{1-q} ~ d x\right)^{\frac{1}{1-q}}+\left(\int_{\Omega}\mathcal{V}^{1-q}  d x\right)^{\frac{1}{1-q}}\right]=1.
\end{aligned}
$$
Moreover, since $h=\frac{v+\mathcal{V}}{2 \zeta} \in \mathcal{Z}$, so  we  see that
$$
\Re(U) \leq\eta(h)^2 \leq\eta\left(\frac{v+\mathcal{V}}{2 \zeta}\right)^2 \leq \frac{\Re(U)}{\zeta^2} \leq \Re(U).
$$
Hence, we get $\zeta=1$ and 
$$
\Re(U)^{\frac{1}{2}}= \eta\left(\frac{v+\mathcal{V}}{2}\right)={\frac{\eta(v)}{2}+\frac{\eta(\mathcal{V})}{2} }
$$
Using the fact that norm $v \rightarrow \eta(v)$ is strictly convex, we have $v=\mathcal{V}$. Moreover,   $\Re(U)=\eta(v)^2$ for some $v \in \mathcal{Z}$, if and only if $v=\mathcal{V}$ or $-\mathcal{V}$. Thus, if 
\begin{equation}
    \eta(v)^2\geq C\left(\int_{\Omega}|v|^{1-q}\,dx\right)^{\frac{2}{1-q}}
\end{equation}
holds for some $w \in\mathcal{X}^{1,2}_{\mathcal{D}}(U) \backslash\{0\}$ and easily we can see that $\tau w \in \mathcal{Z}$, where
$$
\tau=\frac{1}{\left(\int_{\Omega}|w|^{1-q}  d x\right)^{\frac{1}{1-q}}}.
$$
Hence, $\mathcal{V}=\tau w$ or $\mathcal{V}=-\tau w$, since \(\mathcal{V} = \left( \int_\Omega \hat{u}^{1-q}  \, dx \right)^{-\frac{1}{1-q}} \hat{u}\).
Thus, the proof is completed now.
\end{proof}
Following is a direct consequence of Proposition \ref{pp3.3}.
  \begin{Corollary}
      If \(\mathcal{V} =\left( \int_\Omega \hat{u}^{1-q}  \, dx \right)^{-\frac{1}{1-q}} \hat{u} \in \mathcal{Z} \) then
$
\Re(U) = \eta(\mathcal{V})^2$.
\end{Corollary} 

\subsection{Proof of Theorem \ref{thm2.4}:}
\begin{proof}Suppose \eqref{eqmixed} holds. By Lemma \eqref{lem4.3}, the sequence $\{u_n\}$ is uniformly bounded in $\mathcal{X}^{1,2}_{\mathcal{D}}(U)$. Now, following similar arguments as in the proof of Theorem \ref{existence}, then we obtain that problem \eqref{P_q} admits a weak solution in $\mathcal{X}^{1,2}_{\mathcal{D}}(U)$. Conversely, let \(u \in \mathcal{X}^{1,2}_{\mathcal{D}}(U)\) be a weak solution to problem \eqref{P_q}. By applying Proposition \ref{prop3.1} and choosing \(u\) as a test function in \eqref{P_q}, we obtain
\begin{equation}\label{3418}
  \int_{\Omega} |\nabla {u}|^2 \,dx +\int_{Q} \frac{|{u}(x)-(y)|^2}{|x-y|^{N+2s}} dxdy = \int_{\Omega} u^{1-q} \, dx
\implies
 \eta(u)^2 = \int_{\Omega} u^{1-q} \, dx.
\end{equation}
Again, using Proposition \ref{prop3.1} and choosing a \(|v| \in \mathcal{X}^{1,2}_{\mathcal{D}}(U)\) as a test function in \eqref{P_q}, we get
\begin{equation}\label{3419}
\int_{\Omega} |v| u^{-q} \, dx \leq C \eta(u) \eta(|v|)\leq C \eta(u) \eta(v),
\end{equation}
where we used H\"older inequality and \(C\) is a positive constant. Therefore, for every \(v \in \mathcal{X}^{1,2}_{\mathcal{D}}(U)\), using \eqref{3418}, \eqref{3419} and H\"{o}lder inequality again, we have
$$
\begin{aligned}
\int_{\Omega} |v|^{1-q} \, dx = \int_{\Omega} \left( |v| u^{-q} \right)^{1-q} \left( u^{1-q} \right) \, dx
&\leq \left( \int_{\Omega} |v| u^{-q} \, dx \right)^{1-q} \left( \int_{\Omega} u^{1-q} \, dx \right)^{q} C \eta(u)^{q+1} \eta(v)^{1-q}.
    \end{aligned}
$$
Thus, we conclude that \eqref{eqmixed} holds.
\end{proof}

\section{Perturbed singular problem}
We study \eqref{1} for $g(u)= \lambda u^{-q}+u^p$  with  $0<q<1<p\leq 2^*-1$ and devote this section to proof of main result Theorem \ref{2}. The features of the Nehari manifold $\mathcal{N}_{\lambda}$ associated to the functional $\mathcal{J}_{\lambda}$ is discussed in the upcoming section.
\subsection{Analysis of Nehari manifold}
 The following analysis is crucial to prove our existence results. Let us recall that
$$\mathcal{J}_{\lambda}(u)= \frac{1}{2}\eta(u)^2- \frac{\lambda}{1-q} \int_\Omega |u(x)|^{1-q} dx
- \frac{1}{p+1} \int_\Omega |u(x)|^{p+1}dx.$$
Then the Nehari manifold is defined as
 \[
\mathcal{N}_{\lambda} = \bigg\{ u\in \mathcal{X}^{1,2}_{\mathcal{D}}(U) :~    \langle \mathcal{J}'_{\lambda} (u),u\rangle=0\bigg\}
 = \bigg\{ u\in \mathcal{X}^{1,2}_{\mathcal{D}}(U):~   \eta(u)^2 -\lambda  \int_{\Omega} |u(x)|^{1-q} dx
 -  \int_{\Omega}|u(x)|^{p+1} dx = 0\bigg\}.
\]
It is easy to see that $\mathcal{J}_{\lambda}$ fails to be bounded below over the whole space $\mathcal{X}^{1,2}_{\mathcal{D}}(U)$, but the next result is of great utility.
\begin{lemma} \label{thm3.1}
$\mathcal{J}_{\lambda}$ is coercive and bounded below on $\mathcal{N}_{\lambda}$.
\end{lemma}
\begin{proof}
 Since $u\in \mathcal{N}_\lambda$, due to continuous embedding of
 $\mathcal{X}^{1,2}_{\mathcal{D}}(U)$  in $L^{2^*}(U)$, we find
\begin{align*}
\mathcal{J}_{\lambda}(u) = \left(\frac{1}{2}-\frac{1}{p+1} \right)\eta(u)^2
- \lambda \left(\frac{1}{1-q}-\frac{1}{p+1}\right)\int_{\Omega} |u(x)|^{1-q}dx \geq A \eta(u)^2 - B \eta(u)^{1-q},
\end{align*}
for some constants $A,B>0.$ 
 Hence, being $q\in(0,1)$, the functional $\mathcal{J}_\lambda$ is coercive. Now define the map $Z : \mathbb{R}^{+} \to \mathbb{R}$ as
$Z(t)= A_1t^2-A_2t^{1-q}.$
Clearly, $Z$ has a unique point of minimum, since $1-q<2$.  This gives
at once that $\mathcal{J}_{\lambda}$ is bounded below on $\mathcal{N}_{\lambda}.$
\end{proof} 
Now, in order to study the Nehari manifold geometry, let us  introduce
for each $u\in \mathcal{X}^{1,2}_{\mathcal{D}}(U)$, the associated fiber map $\Phi_u:\mathbb{R}^+ \to \mathbb{R}$ defined as
$\Phi_u(t)=\mathcal{J}_{\lambda}(tu)$ {for all $t\in\mathbb{R}^+$. Consequently,}
$$\Phi_u(t)=
		\frac{t^2}{2} \eta(u)^2 - \lambda \frac{t^{1-q}}{1-q} \int_{\Omega} |u(x)|^{1-q} dx
      - \frac{ t^{p+1}}{p+1} \int_{\Omega} |u(x)|^{p+1} dx,
$$
so that
\begin{equation}\label{p1}
\Phi'_u(t) = t \eta(u)^2 - \lambda t^{-q} \int_{\Omega} |u(x)|^{1-q} dx
 - {t^{p}} \int_{\Omega}|u(x)|^{p+1} dx.
     \end{equation}
\begin{equation}\label{p2}
\Phi''_u(t) =\eta(u)^2 +  \lambda q t^{-q-1}
\int_{\Omega} |u(x)|^{1-q} dx -  p t^{p-1}
\int_{\Omega} |u(x)|^{p+1} dx.
    \end{equation}
It is clear that the points in $\mathcal{N}_{\lambda}$ {correspond} to the critical points of $\Phi_{w}$ at $t = 1$. To represent {the local minima,
the local maxima and the points of inflexion of $\Phi_{w}$, the manifold} $\mathcal{N}_{\lambda}$ is splitted into three sets respectively. Accordingly, we define
\begin{align*}
\mathcal{N}_{\lambda}^{+}
&=  \{ u \in \mathcal{N}_{\lambda}:~\Phi'_u (\Bar{t}) = 0,~  \Phi''_u(1) > 0\}\\
&=  \{ \Bar{t}u \in \mathcal{N}_{\lambda} :~ \Bar{t} > 0,~ \Phi'_u (\Bar{t}) = 0,~ \Phi''_u(\Bar{t}) > 0\},\\
\mathcal{N}_{\lambda}^{-}
&=  \{ u \in \mathcal{N}_{\lambda}:~ \Phi'_u (1) = 0,~~ \Phi''_u(1) <0\}\\
&=  \{ \Bar{t}u \in \mathcal{N}_{\lambda} : \Bar{t} > 0,~  \Phi'_u (\Bar{t}) = 0,~ \Phi''_u(\Bar{t}) < 0\}
\end{align*}
and $ \mathcal{N}_{\lambda}^{0}= \{ u \in \mathcal{N}_{\lambda} : \Phi'_{u}(1)=0,
 \Phi''_{u}(1)=0 \}$.
\begin{lemma}\label{ 3.2}
There exists $\lambda^*>0$ such that for all
$\lambda \in (0,\lambda^*)$ and $u \in \mathcal{X}^{1,2}_{\mathcal{D}}(U)$,
there are unique points $t_{\max}$, $t_*$ and $t^*$ satisfying
$t_*<t_{\max}<t^*$, $t_* u\in \mathcal{N}_{\lambda}^{+}$ and
$ t^* u\in \mathcal{N}_{\lambda}^{-}$ for all $\lambda\in (0, \lambda^*)$.
\end{lemma}
\begin{proof} Fix $u \in \mathcal{X}^{1,2}_{\mathcal{D}}(U)$, if we set
 $M(u)= \int_{\Omega}|u(x)|^{1-q}dx~\mbox{and}~
N(u)= \int_{\Omega}|u(x)|^{p+1} dx.$
then
\[
\frac{d}{dt}\mathcal{J}_{\lambda}(tu)
= t\eta(u)^2 - \lambda t^{-q} M(u) - t^{p} N(u)\\
= t^{-q} \big\{\mu_u(t) - \lambda M(u) \big\},
\]
where $\mu_u(t) := t^{1+q} {\eta(u)^2} - t^{p+q} N(u)$.
We can easily verify that $\mu_u(t) \to -\infty$ as  $t \to \infty$ and $\mu_u$
 achieves its maximum at
\[
t_{\max} = \left[ \frac{(1+q)\eta(u)^2}{(p+q) N(u)}
\right]^{\frac{1}{p-1}}
\]
 such that
\[
\mu_u(t_{\max}) = \left( \frac{p-1}{p+q} \right)
\left( \frac{1+q}{p+q} \right)^{\frac{1+q}{p-1}}
\frac{(\eta(u)^2)^\frac{p+q}{p-1}}{(N(u))^\frac{1+q}{p-1}}.
 \]
Since $tu \in \mathcal{N}_{\lambda}$  if and only if  $\mu_u(t) =  \lambda M(u) $, the following is true:
\begin{align*}
\mu_u(t) - \lambda M(u)\geq \mu_u(t_{\max}) - \lambda  \|u\|^{1-q}_{L^{1-q}(\Omega)} =  \Big( \frac{p-1}{p+q} \Big)
\Big( \frac{1+q}{p+q} \Big)^{\frac{1+q}{p-1}}
\frac{(\eta(u)^2
)^\frac{p+q}{p-1}}{{N(u)^\frac{1+q}{p-1}}}
- \lambda  \|u\|^{1-q}_{L^{1-q}(\Omega)}>0
\end{align*}
provided
\begin{equation}\label{lambda^*}
\lambda<
\left( \frac{p-1}{p+q} \right)
\left( \frac{1+q}{p+q} \right)^{\frac{1+q}{p-1}}
  C^{\frac{-1-q}{p-1}} (C_{1-q})^{-1} :=\lambda^*,    
\end{equation}
 where $C_{\alpha} = \sup \Big\{ \dint_{\Omega} |u|^{\alpha} dx :  u\in \mathcal{X}^{1,2}_{\mathcal{D}}(U), \; \eta(u) = 1\Big\}, $ for each $\alpha\geq 0$.
An easy observation tells that $ \mu_u(t) = \lambda \int_{\Omega}  |u|^{1-q}dx$ if and only if $\Phi'_u(t) = 0$. Therefore, whenever $\lambda \in (0,\lambda^*)$ there are exactly two points $t_*,t^*$ such that $0<t_*<t_{\max}<t^*$ with $\mu'_u(t_*)>0$ and $\mu'_u(t^*)<0$ i.e. $t_*u \in \mathcal{N}^{+}_{\lambda}$ and $t^*u \in {\mathcal{N}^{-}_{\lambda}}$. Moreover,  $\Phi_{u}$ is decreasing in $(0,t_*)$ and increasing in  $(t_*,t^*)$ with a local minimum at $t=t_*$ and a local maximum at $t=t^*$.
\end{proof}
\begin{lemma}\label{below}
  It holds that $\inf_{w\in \mathcal{N}_{\lambda}^{+}}\mathcal{J}_{\lambda}(w)<0$.
  \end{lemma}
\begin{proof}
The proof follows similar to [\cite{LSharma}, Lemma 5.1 ].
    \end{proof}
    Let $S>0$ is a embedding constant satisfying $\|u_n\|^2_{L^{2^*}(U)}S \leq \eta(u_n)^2$ by virtue of Remark \ref{r2.5}. The proof of following Lemma follows from Lemma \eqref{ 3.2}.
\begin{lemma}
 Suppose $0<q<1.$ Let $\Lambda$ be a constant defined by
\begin{equation*}
\Lambda := 
\left\{
\begin{aligned}
    & \sup\Big\{\lambda>0:\; \text{for each}\; 0 \leq u \in \mathcal{X}^{1,2}_{\mathcal{D}}(U),\; \Phi_u \;\text{has two critical points in}\; (0,\infty)\Big\}, \quad &\text{if}\; p<2^*-1, \\
    & \sup\Big\{\lambda>0:\; \text{for each}\; 0 \leq u \in \mathcal{X}^{1,2}_{\mathcal{D}}(U),\; \Phi_u \;\text{has a critical point in}\; (0,\infty)\Big., \\
    &\qquad\quad\quad \sup\Big\{ \eta(u): u \in \mathcal{X}^{1,2}_{\mathcal{D}}(U), \Phi^{\prime}(1)=0, \Phi^{''}_u(1)>0 \Big\} \leq \left( \frac{2}{2^*} \right)^{\frac{2}{2^*-2}} S^{\frac{2^*}{2^*-2}}\Big\}, \quad &\text{if}\; p=2^*-1.
\end{aligned}
\right.
\end{equation*}
  Then $\Lambda>0$.   
\end{lemma}
    \begin{Corollary}\label{cor3.3}
It holds that $\mathcal{N}_{\lambda}^{0} = \{0\}$, for each $ \lambda \in (0, \Lambda)$.
\end{Corollary}
\begin{proof}
The proof can be done using contradiction and an immediate application of  Lemma \ref{ 3.2}.
\end{proof}
\textbf{From now on, we fix $ \lambda \in (0, \lambda^*)$,
without further mentioning.} The next result establishes that $\mathcal{N}_{\lambda}^{+}$ and $\mathcal{N}_{\lambda}^{-}$ are bounded in a suitable sense.

 The following result follows from [\cite{LSharma}, Lemma 3.4 ].
\begin{lemma}\label{lem3.4}
 The following properties hold true.
\begin{itemize}
\item[$(P_{1})$] $\sup \{ \eta(w): w\in \mathcal{N}_{\lambda}^{+}\} < \infty $
\item[$(P_{2})$] $\inf \{ \eta(u): u \in \mathcal{N}_{\lambda}^{-} \} >0$  and
$ \sup \{ \eta(u) : u \in \mathcal{N}_{\lambda}^{-}, \mathcal{J}_{\lambda}(u) \leq \mathcal{M}\} < \infty$,
for each $\mathcal{M} > 0$.
\end{itemize}
Furthermore, $\inf \mathcal{J}_{\lambda}(\mathcal{N}_{\lambda}^{+}) > - \infty$ and
$\inf \mathcal{J}_{\lambda}(\mathcal{N}_{\lambda}^{-}) > - \infty$.
\end{lemma}
\subsection{\texorpdfstring{Properties of minimizers over $\mathcal{N_{\lambda}}^{+}$ and $\mathcal{N_{\lambda}}^{-}$}{}}

This section contains essential properties of minimizers of $\mathcal{J}_{\lambda}$ over $\mathcal N_\lambda^+$ and $\mathcal N_\lambda^-$, which {allow us to show that} these minimizers are weak solutions of $(P_\lambda)$.
 The following result follows the same as [\cite{LSharma}, Lemma 4.1 and 4.2].
\begin{lemma} \label{3.5}
Let $u$ and $\tilde{u}$ be two minimizers of $\mathcal{J}_{\lambda}$ over
$\mathcal{N}_{\lambda}^{+}$ and
$\mathcal{N}_{\lambda}^{-}$, respectively. Then for all non negative $   w \in\mathcal{X}^{1,2}_{\mathcal{D}}(U)$
\begin{enumerate}
\item  there is $\epsilon{'} > 0$ such that $\mathcal{J}_{\lambda}(u +\epsilon w) \geq \mathcal{J}_{\lambda}(u)$
 for each $ \epsilon \in [0,\epsilon{'}]$, and

\item  {if} $t_{\epsilon}$
 is the unique positive real number which satisfies
 $t_{\epsilon} (\tilde{u} + \epsilon w) \in \mathcal{N}_{\lambda}^{-}$,
 {then} $t_{\epsilon} \to 1$  as $\epsilon \to 0^+$.
\end{enumerate}
\end{lemma}

\begin{lemma} \label{lem3.6}
Assume that $u$ and $\tilde{u}$ are non negative minimizers of $\mathcal{J}_{\lambda}$ on $\mathcal{N}_{\lambda}^{+}$ and
$\mathcal{N}_{\lambda}^{-}$ respectively. Then
for all $w\in\mathcal{X}^{1,2}_{\mathcal{D}}(U)$, we have $u^{-q}w ~,~ \tilde{u}^{-q} w \in L^{1}(\Omega)$ and
\begin{gather}
\int_{\Omega}\nabla u \cdot \nabla w \, dx + \int_{{Q}}  \frac{(u(x)-u(y)) (w(x) - w(y))} {|x-y|^{N+2s}} \, dx \, dy
 - \lambda \int_\Omega  u^{-q}(x)w(x)\,dx
 - \int_\Omega u^{p}(x)w(x) dx  \geq 0, \label{eq4.2}\\
\int_{\Omega}\nabla \tilde{u} \cdot \nabla w \, dx + \int_{{Q}}  \frac{(\tilde{u}(x)-\tilde{u}(y)) (w(x) - w(y))} {|x-y|^{N+2s}} \, dx \, dy
 - \lambda \int_\Omega  \tilde{u}^{-q}(x)w(x)\,dx
 - \int_\Omega \tilde{u}^{p}(x)w(x) dx \geq 0.\label{eq4.3}
\end{gather}
\end{lemma}

Suppose {that} $\phi_{1}>0$ is the first eigenfunction of $\mathcal{L}$ corresponding to the smallest eigenvalue $\lambda_1$ with the mixed boundary condition. Then $0<\phi_{1}\in L^{\infty}(U) $, {as shown in \cite{Tm}} and
\begin{equation} \label{firstev}
    \left\{
    \begin{aligned}
        \mathcal{L}\phi_1 &= \lambda_1 \phi_1, \quad \text{in} \quad \Omega, \\
        \phi_1 &= 0 \quad \text{in} \quad U^c, \\
        \mathcal{N}_s(\phi_1) &= 0 \quad \text{in} \quad \mathcal{N}, \\
        \frac{\partial \phi_1}{\partial \nu} &= 0 \quad \text{in} \quad \partial \Omega \cap \overline{\mathcal{N}}
    \end{aligned}
    \right.
\end{equation}
and satisfies for all $0\leq \varphi \in \mathcal{X}^{1,2}_{\mathcal{D}}(U)$
\begin{equation}
      \int_{\Omega}\nabla \phi_1 \cdot \nabla \varphi\, dx + \int_{{Q}}  \frac{(\phi_1(x)-\phi_1(y)) (\varphi(x) - \varphi(y))} {|x-y|^{N+2s}} \, dx \, dy 
 = \lambda_1\int_\Omega \phi_1(x)\varphi(x) \,dx.
\end{equation}
For instance, here we assume $\|\phi_1\|_{L^{\infty}(U)}=1.$
Let $\beta > 0$ be such that
$ f = \beta \phi_{1}$
and satisfies $\forall$ $0\leq \varphi \in \mathcal{X}^{1,2}_{\mathcal{D}}(U),$
 \begin{equation} \label{eq3.5}
 -\left(\int_{\Omega}\nabla f \cdot \nabla \varphi\, dx + \int_{{Q}}  \frac{(f(x)-f(y)) (\varphi(x) - \varphi(y))} {|x-y|^{N+2s}} \, dx \, dy\right)
+\int_\Omega  (\lambda f^{-q}(x)+f^{p}(x))\varphi(x) dx  >0,
\end{equation}
 and  $f^{p+q} (x) \leq \lambda\frac{q}{p}$
for a.e. $x \in \R^N$. 
\begin{lemma} \label{lem2.6}
For each non negative $w \in \mathcal{X}^{1,2}_{\mathcal{D}}(U)$, there exists
{a sequence $(w_m)_m$} in $\mathcal{X}^{1,2}_{\mathcal{D}}(U)$
such that each $w_{m}$ is compactly supported in $U$, with $0 \leq w_1 \leq w_2 \leq \ldots$  and $w_m \to w$ strongly in $\mathcal{X}^{1,2}_{\mathcal{D}}(U) $ as $m\to \infty$.
\end{lemma}
\begin{proof}
Let $0\leq w \in \mathcal{X}^{1,2}_{\mathcal{D}}(U)$ and let ${(g_m)_m}$  be a nonnegative  sequence in  ${\mathcal{C}}^{\infty}_{0}(U)$ such that  
  $g_m \to w$  strongly in $\mathcal{X}^{1,2}_{\mathcal{D}}(U)$ as $m\to \infty$.
Define $u_m =  \min \{g_m, w\}$ for each $m$ which strongly converges to $w$ in
$\mathcal{X}^{1,2}_{\mathcal{D}}(U)$. Let $w_1 = u_{s_1}$, choosing $s_1\in \mathbb N$ which satisfies
\[\eta(u_{s_1}-w) \leq 1.\]
Then $\max \{w_1, u_m\} \to w$ strongly in $\mathcal{X}^{1,2}_{\mathcal{D}}(U)$ as
$m \to \infty$.  We now choose $s_2\in \mathbb N$ satisfying
\[\eta(\max\{w_1,u_{s_2}\}-w) \leq 1/2.\]
Let $w_2=\max\{w_1,u_{s_2}\}$ then $\max \{w_2, u_m\} \to w$ strongly as $m \to \infty$.
Proceeding as above, in general, we set
\[w_{m+1}= \max\{w_m,u_{s_{m+1}}\}.\]
It is easy to see that $w_m \in \mathcal{X}^{1,2}_{\mathcal{D}}(U)$  has compact support
for each $m$ and
\[\eta(\max\{w_m,u_{s_{m+1}}\}-w) \leq \frac{1}{m+1}\]
which implies that
{ $(w_m)_m$} converges strongly to $w$ in $\mathcal{X}^{1,2}_{\mathcal{D}}(U)$ as $m \to \infty$.
\end{proof} 

\begin{lemma}\label{le05}
Assume $u$ and $\tilde{u}$ are minimizers of $\mathcal{J}_{\lambda}$ on $\mathcal{N}_{\lambda}^{+}$ and
$\mathcal{N}_{\lambda}^{-}$ respectively. Then there holds $u \geq f$ and $\tilde{u}\geq f$ a.e. in $\R^N$.
 \end{lemma}
\begin{proof}
An easily noticeable fact is

\begin{equation}\label{eq3.4}
\frac{d}{dr}(\lambda r^{-q}+r^{p})
= -q \lambda r^{-q-1} +pr^{p-1} \leq 0 \iff ~ r^{p+q}\leq  \frac{\lambda q}{p}.
\end{equation}
By {Lemma} \ref{lem2.6}, we get existence of a non negative sequence
$\{w_k\}_k$ in $\mathcal{X}^{1,2}_{\mathcal{D}}(U)$  such that $0\leq w_k\leq (f-u)^+$, each $w_k$ is compactly supported in $U$ and as $k\to \infty$
\[w_k\to (f-u)^+\;\text{strongly in } \mathcal{X}^{1,2}_{\mathcal{D}}(U).\]
Thus,   \eqref{eq4.2} and \eqref{eq3.5} implies that
\begin{align*}
&\int_{{\Omega}} \Big ( \nabla (u-f)\cdot \nabla w_{k} \Big ) \,dx+ \int_{{Q}}  \frac{\left((u-f)(x)-(u-f)(y)\right) (w_k(x) - w_k(y))} {|x-y|^{N+2s}} \, dx \, dy \\
&- \int_\Omega \Big(\lambda u^{-q}(x)  + u^{p}(x)\Big)w_k(x)\,dx
+ \int_{\Omega} \Big(\lambda f^{-q}(x) +f^{p}(x)\Big)w_k(x)\,dx \geq 0.
\end{align*}
Since obviously $w_k(x)\to (f- u)^+(x)$ pointwise a.e. in $\R^N$, we set
$w_k(x) = (f- u)^+(x)+ o(1)$ as $k \to \infty$. Now considering the first and second terms in above, we see that
\begin{equation}\label{eq427}
\int_{{\Omega}} \Big ( \nabla (u-f)\cdot \nabla w_{k} \Big ) \,dx = \int_{{\Omega}} \Big ( \nabla (u-f)\cdot  \nabla(f-u)^+  \Big ) \,dx,
\end{equation}
and 
\begin{equation} \label{eq428}
\begin{aligned}
    &\int_{{Q}} \frac{\left((u-f)(x)-(u-f)(y)\right) (w_k(x) - w_k(y))} {|x-y|^{N+2s}} \, dx \, dy \\
    &\quad = \int_{{Q}} \frac{\left((u-f)(x)-(u-f)(y)\right) ((f-u)^+(x) - (f-u)^+(y))} {|x-y|^{N+2s}} \, dx \, dy  + o(1) \int_{{Q}} \frac{\left((u-f)(x)-(u-f)(y)\right)} {|x-y|^{N+2s}} \, dx \, dy.
\end{aligned}
\end{equation}
       Now, using $$\nabla(f -u)(x)=\nabla((f-u)^{+}(x)-(f-u)^{-}(x))$$ and from \eqref{eq427} and \eqref{eq428}, we can see
\begin{equation}\label{nbla}
    \begin{aligned}
     \int_{{\Omega}} \Big ( \nabla (u-f)\cdot  \nabla(f-u)^+  \Big ) \,dx
= - \| \nabla(f - u)^{+}\|^{2}_{L^2(\Omega)}
\end{aligned}
\end{equation}
and
\begin{equation}\label{eq429}
 \begin{aligned}
&\int_{{Q}}  \frac{\big((u-f)(x)-(u-f)(y)\big) \big((f-u)^+(x) - (f-u)^+(y)\big)} {|x-y|^{N+2s}} \, dx \, dy \\
&\quad= \int_{{Q}}  \frac{\Big[\big((f-u)^+(y)-(f-u)^-(y)\big)- \big((f-u)^+(x) - (f-u)^-(x)\big)\Big] \big((f-u)^+(x) - (f-u)^+(y)\big)} {|x-y|^{N+2s}} \, dx \, dy\\
&\quad\leq \int_{{Q}}  \frac{\left|(f-u)^+(x)-(f-u)^+(y)\right|^2}  {|x-y|^{N+2s}} \, dx \, dy. 
    \end{aligned}
       \end{equation}
Since $f^{p+q} (x) \leq  \frac{\lambda q}{p}$
a.e. $x \in \R^N$, using \eqref{eq3.4} we get
\begin{equation}\label{eq4210}
\begin{aligned}
& \int_\Omega \Big((\lambda u^{-q}(x) + u^{p}(x))
 - (\lambda f^{-q}(x) + f^{p}(x))\Big)w_k(x)\,dx \\
& = \int_{\Omega \cap \{f \geq u\}} \Big((\lambda u^{-q}(x) + u^{p}(x))
 - (\lambda f^{-q}(x) + f^{p}(x))\Big)((f-u)^{+}(x)+o(1))\,dx  \geq 0.
    \end{aligned}
 \end{equation}
Hence, from \eqref{nbla}, \eqref{eq429} and \eqref{eq4210}, we deduce that
 \begin{align*}
 0 &\leq -\eta\left((f-u )^+)\right)^2 -\int_\Omega (\lambda (u^{-q}(x) + u^{p}(x))w_k(x)\,dx+ \int_{\Omega} (\lambda f^{-q}(x) + f^{p}(x))w_k(x)\,dx +o(1)\\
 &\leq -\eta\left((f-u )^+)\right)^2+ o(1)
  \end{align*}
As $k \to \infty$, this gives
$\eta\left((f-u )^+)\right)^2 \leq 0$ which says that $u\geq f$ a.e. in $\R^N$. A similar argument shows that $\tilde{u}\geq f$ a.e. in $\R^N$.
\end{proof} \medskip
\subsection{Proof of Theorem \ref{2}:}
This subsection contains the most crucial part of this section, i.e. the proof of achieving $\inf \mathcal{J}_{\lambda}(\mathcal N_\lambda^+)$ and $\inf \mathcal{J}_{\lambda}(\mathcal N_\lambda^-)$. Also, it contains the proof that these minimizers are weak solutions of $(P_\lambda)$.


 \begin{lemma}
Let $\lambda \in\left(0, \Lambda\right), 0<q<1<p<2^*-1$.  Then there exist $u_{\lambda} \in \mathcal{N}_{\lambda}^{+}$and $v_{\lambda} \in \mathcal{N}_{\lambda}^{-}$ satisfying
$$
\mathcal{J}_{\lambda}\left(u_{\lambda}\right)=\inf_{u\in \mathcal{N}_{\lambda}^{+}} \mathcal{J}_{\lambda}(u) \quad \text{and}~~~ \mathcal{J}_{\lambda}\left(v_{\lambda}\right)=\inf_{u\in \mathcal{N}_{\lambda}^{-}} \mathcal{J}_{\lambda}(u).
$$
            \end{lemma}
\begin{proof}
Firstly, we will show that there exists $u_{\lambda} \in \mathcal{N}_{\lambda}^{+}$ such that $\mathcal{J}_{\lambda}\left(u_{\lambda}\right)=\inf_{u\in \mathcal{N}_{\lambda}^{+}} \mathcal{J}_{\lambda}(u)$. By Lemma \ref{thm3.1}, $\mathcal{J}_{\lambda}$ is bounded below on $\mathcal{N}_{\lambda}^{+}$.
Let  $\{u_{n}\} \subset \mathcal{N}_{\lambda}^{+}$ be a sequence such that $\mathcal{J}_{\lambda}(u_{n}) \to \inf_{u\in \mathcal{N}_{\lambda}^{+}} \mathcal{J}_{\lambda}(u)$ as $n \to \infty$.   Using condition $(P_1)$ of Lemma \ref{lem3.4} yields that there exists $u_{\lambda} \in \mathcal{X}^{1,2}_D(U) $  such that
$u_{n} \rightharpoonup u_{\lambda}$ weakly in the space $\mathcal{X}^{1,2}_D(U)$, up to a subsequence if necessary.
We claim that   $u_n \to u_{\lambda}$ strongly in $\mathcal{X}^{1,2}_D(U),$
as $n\to \infty$. Otherwise, we may assume $$\eta(u_n-u_{\lambda})^2 \to a^2 \neq 0$$
as $n \to \infty$. 
By {the} Brezis-Lieb lemma we have
 $$\lim_{n \to \infty}\Big(\eta(u_{\lambda})^2 -\eta(u_n)^2  + \eta(u_{\lambda}- u_n)^2 \Big) = 0.$$
Since $u_n\in\mathcal{N}_{\lambda}^{+}$ and using embedding result (see Remark \ref{r2.5}) for  $\mathcal{X}^{1,2}_D(U)$ in the subcritical case, we obtain
\begin{equation}\label{eq8}
0  = \lim_{n \to \infty} \Phi'_{u_n}(1) = \Phi'_{u_{\lambda}}(1)+a^2,
\end{equation}
which implies
$$
\eta(u_{\lambda})^2 + a^2 = \lambda \int_{\Omega}|u_{\lambda}(x)|^{1-q}dx + \int_{\Omega}|u_{\lambda}(x)|^{p+1}dx.
$$
Before moving further, we claim that $u_{\lambda} \in \mathcal{X}^{1,2}_D(U)$ can not be identically zero. If $u_{\lambda}\equiv 0$
 and $u_n\to u_{\lambda}$ strongly in $\mathcal{X}^{1,2}_D(U)$, then Lemma \ref{below} gives $0=\mathcal{J}_{\lambda}(0)=\mathcal{J}_{\lambda}(u_{\lambda})\leq {\operatorname{\liminf}}_{n\to\infty} \mathcal{J}_{\lambda}(u_n)<0,$
giving a contradiction. 
Now applying Lemma \ref{ 3.2}, we state that there exists
{$t_*$, $t^*$, with} $0 < t_* < t^*$ such that
$\Phi'_{u_{\lambda}}(t_*)= \Phi'_{u_{\lambda}}(t^*) = 0$ and
$t_{*}u_{\lambda} \in \mathcal{N}^{+}_{\lambda}$, $t^*u_{\lambda} \in \mathcal{N}^{-}_{\lambda}$. From \eqref{eq8}, we have $\Phi^{\prime}_{u_{\lambda}}(1)<0$, which gives the following two cases as follows.
\smallskip

\noindent\textbf{Case (i) ($t^* < 1$)} We consider the function
\[
g(t) = \Phi_{u_{\lambda}}(t)+ \frac{a^2t^2}{2},
\]
for $t >0$. From equation \eqref{eq8}, we get $ g'(1) = \Phi'_{u_{\lambda}}(1)+a^2 = 0$
and
\[
 g'(t^*) = \Phi'_{u_{\lambda}}(t^*)+t^*a^2
= {t^*}a^2
 > 0
\]
from which we can conclude that $g$ is
{an} increasing function on $[t^*,1]$. Hence,
\begin{align*}
\inf_{u\in \mathcal{N}_{\lambda}^{+}} \mathcal{J}_{\lambda}(u)
= \lim_{n \to \infty}\mathcal{J}_{\lambda}(u_n) \geq \Phi_{u_{\lambda}}(1) + \frac{a^2}{2}
& = g(1) > g(t^*) =\Phi_{u_{\lambda}}(t^*) + \frac{a^2(t^*)^{2}}{2}
\geq \Phi_{u_{\lambda}}(t^*) + \frac{(t^*)^{2}}{2} a^2 \\
& > \Phi_{u_{\lambda}}(t^*) > \Phi_{u_{\lambda}}(t_*)
\geq \inf_{u\in \mathcal{N}_{\lambda}^{+}} \mathcal{J}_{\lambda}(u),
\end{align*}
which is a contradiction.
\smallskip
Thus the only case that possibly holds is the following-\\
\noindent \textbf{Case (ii) ($t_* \geq 1$ )} Here, by \eqref{eq8}  we have
$$
\inf_{u\in \mathcal{N}_{\lambda}^{+}} \mathcal{J}_{\lambda}(u)
= \mathcal{J}_{\lambda}(u_{\lambda})+\frac{a^2}{2}
\geq \mathcal{J}_{\lambda}(u_{\lambda}) = \Phi_{u_{\lambda}}(1) \geq \Phi_{u_{\lambda}}(t_*)
\geq \inf \mathcal{J}_{\lambda}(\mathcal{N}^+_{\lambda}) .
$$
Clearly, this is possible when $t_* = 1$ and $a^2 = 0$
which yields $a=0$ by equation $\eqref{eq8}$ and $u_{\lambda} \in \mathcal{N}^+_{\lambda}$.  
This shows that $u_n \to u_{\lambda}$ strongly in $\mathcal{X}^{1,2}_D(U)$. Therefore,
$\mathcal{J}_{\lambda}(u_{\lambda}) = \inf_{u\in \mathcal{N}_{\lambda}^{+}} \mathcal{J}_{\lambda}(u)$.

Next, we will show that
 $\{v_n\}$ be a sequence in $\mathcal{N}^-_{\lambda}$ such that
$\mathcal{J}_{\lambda}(v_n) \to \inf \mathcal{J}_{\lambda}(\mathcal{N}^-_{\lambda})= \mathcal{J}_{\lambda}(v_{\lambda})$ as $n \to \infty$. Then using condition $(P_2)$ of Lemma \ref{lem3.4}, we know that there {exists} $v_{\lambda}\in \mathcal{X}^{1,2}_D(U)$ such that
$v_n \rightharpoonup v_{\lambda}$ weakly in $\mathcal{X}^{1,2}_D(U)$ as $n\to \infty$
up to a subsequence if necessary. 
We claim that $v_n \to v_{\lambda}$ strongly in $\mathcal{X}^{1,2}_D(U)$. Assume that $\eta(v_n - v_{\lambda})^2 \to a^2$
, as $n \to \infty$.
 Using the Brezis-Lieb lemma, we obtain
$$\lim_{n \to \infty}\Big(\eta(v_{\lambda})^2 -\eta(v_n)^2  + \eta(v_{\lambda}- v_n)^2 \Big) = 0.$$
Since $v_n\in\mathcal{N}_{\lambda}^{-}$, we obtain
\begin{equation}\label{eq9}
0  = \lim_{n \to \infty} \Phi'_{v_n}(1) = \Phi'_{v_{\lambda}}(1)+a^2.
\end{equation}
which implies
\begin{equation} \label{eq5.1}
    \eta(v_{\lambda})^2 +a^2 = \lambda\int_{\Omega}|v_{\lambda}(x)|^{1-q}dx
+ \int_{\Omega}|v_{\lambda} (x)|^{p+1}.
\end{equation}
By using \eqref{eq5.1}, we have that
\[
\inf \mathcal{J}_{\lambda}(\mathcal{N}^-_{\lambda}) = \lim_{n\to \infty} \mathcal{J}_{\lambda}(v_n)\geq \mathcal{J}_{\lambda}(v_{\lambda})+
\frac{a^2}{2}.
\]
Initially, we now claim that $v_{\lambda}\in \mathcal{X}^{1,2}_D(U) $ is not identically zero in $\R^N$. If $v_{\lambda}\equiv 0$ in $\R^N$
and $a \neq 0$ 
and $\{v_n\}$ converges strongly to $0,$ which gives a contradiction with condition $(P_2)$ of Lemma \ref{lem3.4}. 

Applying Lemma \ref{ 3.2}, we state that there exist
$t_*,\, t^*$, with
$0 < t_* < t^*$, such that
$\Phi'_{v_{\lambda}}(t_*)= \Phi'_{v_{\lambda}}(t^*) = 0$ and
$t_{*}v_{\lambda} \in \mathcal{N}^{+}_{\lambda}$, $t^*v_{\lambda} \in \mathcal{N}^{-}_{\lambda}$.
We consider the functions
$X,Y :(0,\infty) \to \mathbb{R}~~ \text{ defined by}$
\begin{equation}\label{eq11}
 X(t) = \frac{a^2t^2}{2} \quad \text{and} \quad
 Y(t) = \Phi_{v_{\lambda}}(t)+ X(t).
\end{equation}
Then,  two possible cases {arise} as follows.\\
\textbf{Case (i)} $(t^* <1)$  We have $Y'(1)= \Phi'_{v_{\lambda}}(1)+ X'(1) = 0$, using
\eqref{eq11} and $Y'(t^*)= \Phi'_{v_{\lambda}}(t^*)+ X'(t^*)
= t^*a^2 >0$.
We can observe that $Y$ is increasing on $[t^*,1]$ and we have
\[
\inf \mathcal{J}_{\lambda}(\mathcal{N}^-_{\lambda}) \geq Y(1)> Y(t^*)
\geq \mathcal{J}_{\lambda}(t^* v_{\lambda})+\frac{{t^*}}{2}a^2 > \mathcal{J}_{\lambda}(t^* v_{\lambda})
\geq \inf \mathcal{J}_{\lambda}(\mathcal{N}^-_{\lambda}),
\]
which is a contradiction. \\
\textbf{Case (ii)} ($t^* \geq 1$)  For this case,  we are left with the choice that
if $a \neq 0$, then $\Phi'_{v_{\lambda}}(1) = -a^2 <0$ and
$\Phi''_{v_{\lambda}}(1) = -a^2<0$, which contradicts $t^* \geq 1$.
Thus, $a=0$ and $\{v_n\}$ converges strongly to $v_{\lambda}$ in $\mathcal{X}^{1,2}_D(U).$  By assumption $\lambda\in (0,\Lambda)$, we have $\Phi^{''}_{v_{\lambda}}(1)<0.$  Consequently,  $v_{\lambda} \in \mathcal{N}^-_{\lambda}$ and $\mathcal{J}_{\lambda}(v_{\lambda})= \operatorname{\inf} \mathcal{J}_{\lambda}(\mathcal{N}^-_{\lambda}). $ 
\end{proof}
 \begin{lemma}
Let $\lambda \in\left(0, \Lambda\right), 0<q<1, ~p=2^*-1$.  Then there exist $u_{\lambda} \in \mathcal{N}_{\lambda}^{+}$ satisfying
$$
\mathcal{J}_{\lambda}\left(u_{\lambda}\right)=\inf_{u\in \mathcal{N}_{\lambda}^{+}} \mathcal{J}_{\lambda}(u).
$$
         \end{lemma}
    \begin{proof}
The proof follows using identical arguments as [\cite{LSharma}, Proposition 5.2].
\end{proof}

Now, to complete the proof of main Theorem \ref{2}, we claim that  $u_\lambda$ and $v_{\lambda}$ are positive weak solutions of \eqref{1}.
 Fix $\zeta \in C^{\infty}_0(\Omega)$.  By Lemma \ref{le05} and $f>0$ in $\Omega$, we infer that $u_\lambda,~v_\lambda>0$ a.e. in $\Omega$ and  there exists   
 $ M' >0$ such that $u_\lambda \geq M'>0$ and $v_{\lambda}\geq M'>0$ on
the support of $\zeta$.
Then $u_{\lambda}+\epsilon \zeta\ge 0$ for  $\epsilon$
small enough. {The} same arguments as in the proof of
Lemma \ref{3.5} {shows that} $ \mathcal{J}_{\lambda}(u_\lambda+\epsilon \zeta) \geq \mathcal{J}_{\lambda}(u_\lambda)$
for sufficiently small $\epsilon >0$ and $t_{\epsilon} \to 1$ as $\epsilon \to 0^+$, where $t_{\epsilon}$ is a unique positive real number which is satisfying $t_{\epsilon}\left(v_{\lambda}+\epsilon\zeta\right)\in \mathcal{N}_{\lambda}^-.$ Hence, we obtain
\begin{align*}
0 & \leq \lim _{\epsilon \to 0^+} \frac{\mathcal{J}_{\lambda}(u_\lambda+\epsilon\zeta) - \mathcal{J}_{\lambda}(u_\lambda)}{\epsilon}\\
&=  \int_{\Omega}\nabla u_{\lambda} \cdot \nabla \zeta \, dx + \int_{{Q}}  \frac{(u_{\lambda}(x)-u_{\lambda}(y)) (\zeta(x) - \zeta(y))} {|x-y|^{N+2s}} \, dx \, dy 
 - \lambda \int_\Omega  u_{\lambda}^{-q}(x){\zeta(x)}\,dx
 - \int_\Omega u_{\lambda}^{p}(x){\zeta(x)}dx  \,
\end{align*}
and 
\begin{align*}
0 &\leq \lim _{\epsilon \to 0^+} \frac{\mathcal{J}_{\lambda}\left( t_{\epsilon}(v_\lambda+\epsilon\zeta)\right) - \mathcal{J}_{\lambda}(v_\lambda)}{\epsilon}\leq \lim _{\epsilon \to 0} \frac{\mathcal{J}_{\lambda}\left( t_{\epsilon}(v_\lambda+\epsilon\zeta)\right) - \mathcal{J}_{\lambda}\left(t_{\epsilon}v_\lambda\right)}{\epsilon}\\
&= \int_{\Omega}\nabla v_{\lambda} \cdot \nabla \zeta \, dx + \int_{{Q}}  \frac{(v_{\lambda}(x)-v_{\lambda}(y)) (\zeta(x) - \zeta(y))} {|x-y|^{N+2s}} \, dx \, dy 
 - \lambda \int_\Omega  v_{\lambda}^{-q}(x){\zeta(x)}\,dx
 - \int_\Omega v_{\lambda}^{p}(x){\zeta(x)}dx  \,
\end{align*}
Since $\zeta \in C^{\infty}_0({\Omega})$ is arbitrary, this
implies that $u_\lambda$ and $v_{\lambda}$
are  positive weak solution of \eqref{1}. 

As a consequence, we have the following.
\begin{lemma} We have
$\Lambda < \infty.$
    \end{lemma}
\begin{proof}
Taking $\phi_1$ as the test function in \eqref{1}, we obtain
\[
\lambda_1 \int_{\Omega} u \phi_1\,dx =\int_{\Omega} \mathcal{L}u \phi_1 \, dx= \int_{\Omega} (\lambda u^{-q}+ u^p) \phi_1\,dx.\]
Now, we choose $\mu^* > 0$ large enough such that $\mu^* t^{-q} + t^{p} > (\lambda_1 + \epsilon)t$ for all $t > 0$ and $\epsilon>0$. Then 
\begin{equation}\label{eq435}
\int_{\Omega} (\lambda_1+\epsilon)u\phi_1\,dx> \int_{\Omega}\lambda_1 u\phi_1\,dx=\int_{\Omega}(\lambda u^{-q}+ u^p) \phi_1\,dx.
    \end{equation}
which implies
\begin{equation*}
\int_{\Omega}\left( \mu^* u^{-q} +u^p\right)\phi_1 \,dx > \int_{\Omega} \left( \lambda u^{-q} + u^p\right)\phi_1\,dx .   \end{equation*}
Hence
\[\int_{\Omega}(\lambda- \mu^*) u^{-q}\phi_1\,dx<0\]
but as we know that $u>0$, $\phi_1>0$ in $\Omega.$ So, it must be $\lambda<\mu^*.$
Thus, the proof follows. 
    \end{proof}

We shall prove a short regularity result for positive weak solution to problem \eqref{1}.
\begin{lemma} \label{lem6.1}
Let $w$ be a positive weak solution of \eqref{1}
{and let $u \in \mathcal{X}^{1,2}_{\mathcal{D}}(U)$. Then} $w^{-q}u \in L^{1}({\Omega})$ and 
\begin{gather*}
 \int_{\Omega}\nabla u \cdot \nabla w \, dx + \int_{{Q}}  \frac{(u(x)-u(y)) (w(x) - w(y))} {|x-y|^{N+2s}} \, dx \, dy
 - \int_\Omega \big(\lambda  w^{-q}(x)  + w^{2^*-1}(x)\big)u(x) \,dx = 0,
 \end{gather*}
 for each $u \in \mathcal{X}^{1,2}_{\mathcal{D}}(U)$.
\end{lemma}
\begin{proof}
Let $w$ be a  weak solution of \eqref{1} and let
$u \in \mathcal{X}^{1,2}_{\mathcal{D}}(U)$.
Consider first the case
when $u>0$ in  $\Omega$. Then
by Lemma $\ref{lem2.6}$ there exists a sequence ${(u_m)_m \subset} \mathcal{X}^{1,2}_{\mathcal{D}}(U)$ such that
$u_m \to u$ strongly in $\mathcal{X}^{1,2}_{\mathcal{D}}(U)$, where each $u_m$ has compact support in
$U$ and $0 \leq u_1 \leq u_2 \leq \ldots$. Since each $u_m$ has compact support in $U$ and $w$ is a positive weak
solution of \eqref{1}, we obtain
$$\int_{\Omega}\nabla u_m \cdot \nabla w \, dx + \int_{{Q}}  \frac{(u_m(x)-u_m(y)) (w(x) - w(y))} {|x-y|^{N+2s}} \, dx \, dy=\int_\Omega \big(\lambda  w^{-q}(x)u_m(x) + w^{2^*-1}u_m(x)\big) \,dx,$$
for each~ $m\in \mathbb{N}.$
 Using the monotone convergence theorem, we obtain
 $$w^{-q}(x)u(x) \in L^{1}({\Omega})~\text{and}$$
$$\int_{\Omega}\nabla u \cdot \nabla w \, dx + \int_{{Q}}  \frac{(u(x)-u(y)) (w(x) - w(y))} {|x-y|^{N+2s}} \, dx \, dy=\int_\Omega \lambda  w^{-q}(x)u(x) \,dx +\int_\Omega w^{2^*-1}(x)u(x)dx.$$

In the general case, $u = u^+ - u^-$ and $u^+, u^- \in \mathcal{X}^{1,2}_{\mathcal{D}}(U)$.
The proof of the first part shows that the assertion of the lemma holds
for $u^+$ and $u^-$ and so for $u$. This concludes the proof.
\end{proof} 

\textbf{Proof of Theorem \ref{3}:}
  The proof of $L^\infty$-regularity for solutions of $(P_\lambda)$ when $g(u)=u^{-q}$ follows from Theorem \ref{proof-mt3}.
     So let  $u$ be a positive weak solution to \eqref{1} when $g(u)=\lambda u^{-q}+u^{p},$ with $0<q<1<p\leq 2^*-1$ and $\lambda>0$, now we aim to show that $u\in L^{\infty}(U).$   To prove this result, we apply the classical Moser iteration method and limiting arguments of $L^p$ norm.

Let us define 
$$ \varphi(\sigma)=\varphi_{T, \beta}(\sigma)\begin{cases}  
    0 & ~~~\text{if}~~ \sigma\leq 0\\ 
  \sigma^{\beta}&~~~~\text{if} ~~0<\sigma<T,\\
 \beta T^{{\beta}-1} (\sigma-T)+T^{\beta}&~~ ~~\text{if} ~~\sigma\geq T,
      \end{cases}
$$
where $\beta\geq 1$ and $T>0$ is large. By the definition of $\varphi(\sigma)$, it is easy to see that $\varphi$ is a Lipschitz function with Lipschitz constant $L= \beta T^{{\beta}-1}$. Moreover,
we observe that $\varphi(0)=0$, $\varphi\in C^1(\R,\R)$, is a positive convex function and its first derivative is Lipschitz continuous. So,  utilizing the proof of Lemma 3.1 in \cite{JDE}, it is easy to see that if $u\in \mathcal{X}^{1,2}_{\mathcal{D}}(U)$ then $\varphi(u), \varphi(u)\varphi^{\prime}(u)\in \mathcal{X}^{1,2}_{\mathcal{D}}(U).$
Thus, we can see
\begin{equation}\label{501}
\eta(\varphi(u))^2\leq L^2 \eta(u)^2,
    \end{equation} 
    and since $\varphi(u)\in \mathcal{X}^{1,2}_{\mathcal{D}}(U)$ then using Sobolev embedding (see Remark \ref{r2.5}), we have
$$
    \|\varphi(u)\|^2_{L^{p+1}(\Omega)}\leq S \eta(\varphi(u))^2, ~~1\leq p\leq 2^*-1
$$
where $S>0$ is an embedding constant. From \eqref{501}, we have 
\begin{equation}\label{502}
    \|\varphi(u)\|^2_{L^{p+1}(\Omega)}\leq C \eta(u)^2, 
\end{equation}
for some constant $C>0$, where $C=SL^2$. In addition, using the convexity of $\varphi$, for any $x,y\in\R^N$, we obtain
$$\varphi\left(u(x)\right)-\varphi\left(u(y)\right)\leq \varphi^{\prime}\left(u(x)\right)\left( u(x)-u(y)\right),$$
then one has
\begin{equation}\label{503}
\int_{\Omega} \left(\mathcal{L}\varphi(u)\right) \varphi(u)\,dx \leq \int_{\Omega} \left(\mathcal{L}u\right) \varphi^{\prime}(u) \varphi(u)\,dx = \int_{\Omega} \left(\lambda u^{-q} + u^p\right) \varphi^{\prime}(u) \varphi(u)\,dx.
\end{equation}
Therefore, using \eqref{501}, \eqref{502}, and \eqref{503}, we deduce       
\begin{equation}\label{504}
   \|\varphi(u)\|^2_{L^{p+1}(\Omega)}\leq C \int_{\Omega} \left(\lambda u^{-q} + u^p\right) \varphi^{\prime}(u) \varphi(u)\,dx. 
\end{equation}
Using the estimates $\varphi(u)\leq u^{\beta}$, $\varphi^{\prime}(u)\leq\beta (1+ \varphi(u))$, and $u\varphi^{\prime}(u)\leq\beta\varphi(u)$,  we get
\begin{equation}\label{505}
\begin{aligned}
 \int_{\Omega} \left( \lambda u^{-q} + u^p\right) \varphi^{\prime}(u) \varphi(u)\,dx &= \int_{\Omega} \left(\lambda \varphi^{\prime}(u) \varphi(u) u^{-q} + \varphi^{\prime}(u) \varphi(u) u^p\right) \,dx \\ 
 &\leq \lambda \beta \int_{\Omega} \varphi(u) u^{-q} (1+\varphi(u))\,dx + \beta\int_{\Omega} \left(\varphi(u)\right)^2 u^{p-1}\,dx.
\end{aligned}
    \end{equation}
Hence, from \eqref{504} and \eqref{505}, we obtain
\begin{equation}\label{506}
 \left(\int_{\Omega} (\varphi(u))^{p+1} \,dx\right)^{\frac{2}{p+1}}\leq C_1 \beta \left( \int_{\Omega} \left(\varphi(u) u^{-q} + \left(\varphi(u)\right)^2 u^{-q}\right) \,dx+ \int_{\Omega} \left(\varphi(u)\right)^2 u^{p-1}\,dx\right)  
   \end{equation}
where $C_1= C \max\{\lambda,1\}>0.$

\textbf{Claim:} Let $\beta_1>0$ be such that $2\beta_1= p+1$, then $u\in L^{(p+1)\beta_1}(\Omega).$ 

Fixing some $R>0$ (large), whose appropriate value is to be determined later and using H\"older inequality with $p=\beta_1= \frac{p+1}{2}$ and $p'=\frac{p+1}{p-1}$ (conjugate of $p)$, we obtain
\begin{equation}\label{507}
\begin{aligned}
\int_{\Omega}(\varphi(u))^2 u^{p-1}\,dx & =\int_{\{u \leq R\}}(\varphi(u))^2 u^{p-1} \,dx+\int_{\{u>R\}}(\varphi(u))^2 u^{p-1} \,dx \\
& \leq R^{p-1} \int_{\{u \leq R\}}(\varphi(u))^2 \,dx +\left(\int_{\Omega}(\varphi(u))^{p+1} \,dx\right)^{2 / p+1}\left(\int_{\{u>R\}} u^{p+1} \,dx\right)^{\left(p-1\right) / p+1}.
\end{aligned}
    \end{equation}
By the Monotone Convergence theorem, one can choose $R>0$ large enough such that
\begin{equation}\label{508}
\left(\int_{\{u>R\}} u^{p+1} \,dx\right)^{\left(p-1\right) / p+1}  \leq \frac{1}{2 C_1 \beta_1}
    \end{equation}
and using \eqref{508}, \eqref{507} in \eqref{506}, we deduce
\begin{equation}\label{509}
\begin{aligned}
 \left(\int_{\Omega}(\varphi(u))^{p+1} \,dx\right)^{2 / p+1}  \leq 2 C_1 \beta_1\left(\int_{\Omega}(\varphi(u)) u^{-q} \,dx +\int_{\Omega}(\varphi(u))^2 u^{-q} \,dx +R^{p-1} \int_{\{u \leq R\}}(\varphi(u))^2 \,dx\right).
\end{aligned}
    \end{equation}

Now, using $ \varphi_{T, \beta_1}(u)=\varphi(u) \leq u^{\beta_1}$ on the right hand side of \eqref{509} and then letting $T \rightarrow \infty$ on the left hand side, since $2\beta_1= p+1$, we deduce that
\begin{equation}\label{be}
\left(\int_{\Omega} u^{(p+1) \beta_1}\right)^{2 / p+1} \,dx \leq 2 C_1 \beta_1\left(\int_{\Omega} u^{\frac{p+1}{2}-q} \,dx +\int_{\Omega} u^{p+1-q} \,dx +R^{p-1} \int_{\Omega} u^{p+1} \,dx\right)<+\infty.
    \end{equation}
 This proves the claim. Now, we shall find an increasing unbounded sequence $\beta_n$ such that $u\in L^{(p+1)\beta_n}(\Omega)$, for each $n\geq 1$. So, from \eqref{506}, using $\varphi(u) \leq u^\beta$ on the right  hand side and then letting $T \rightarrow \infty$ on the right hand side, we get
\begin{equation}\label{5010}
    \begin{aligned}
\left(\int_{\Omega} u^{(p+1) \beta}\right)^{2 /(p+1)} & \leq C_1 \beta\left(\int_{\Omega} u^{\beta-q}+\int_{\Omega} u^{2 \beta-q}+R^{p-1} \int_{\Omega} u^{2 \beta+p-1}\right) \\
& \leq C_1^{\prime} \beta\left(1+\int_{\Omega} u^{2 \beta-q}+R^{p-1} \int_{\Omega} u^{2 \beta+p-1}\right)
\end{aligned}
\end{equation}
for some constant $C_1^{\prime}>0$. Now we can see 
$$
\int_{\Omega} u^{2 \beta-q}=\int_{u \geq 1} u^{2 \beta-q}+\int_{u<1} u^{2 \beta-q} \leq \int_{u \geq 1} u^{2 \beta+p-1}+|\Omega| .
$$

Using this in \eqref{5010}, with some simplifications, we obtain
$$
\left(1+\int_{\Omega} u^{(p+1) \beta}\right)^{\frac{(\beta-1)}{2 }} \leq C_\beta^{\frac{1}{2(\beta-1)}}\left(1+\int_{\Omega} u^{2 \beta+p-1}\right)^{\frac{1}{2(\beta-1)}}
$$
where $C_\beta=4 C_1^{\prime} \beta(1+|\Omega|)$. For $n \geq 1$, let us define $\beta_{n+1}$ inductively by
$$
2 \beta_{n+1}+p-1=(p+1) \beta_n
$$
that is
$$\beta_{n+1}-1=\frac{p+1}{2}(\beta_n-1)=\left(\frac{p+1}{2}\right)^{m}(\beta_1 -1).$$

Now, we follow the same ideas of the proof of Theorem 1.1 in  \cite{Su}, we infer that
$$
\left(1+\int_{\Omega}u^{(p+1) \beta_{n+1}} d x\right)^{\frac{1}{(p+1)\left(\beta_{n+1}-1\right)}} \leqslant\left(C_{ \beta_{n+1}}\right)^{\frac{1}{2\left(\beta_{n+1}-1\right)}}\left(1+\int_{\Omega}u^{(p+1) \beta_n} d x\right)^{\frac{1}{(p+1)\left(\beta_n-1\right)}}
$$
where $C_{\beta_{{n+1}}}= 4C_1'\beta_{n+1}(1+|\Omega|).$
We now define $C_{n+1}=C_{\beta_{{n+1}}}$ and
$$
B_{n}=\left(1+\int_{\Omega}u^{(p+1) \beta_n} d x\right)^{\frac{1}{(p+1)(\beta_{n}-1)}}.
$$

In particular, note that $B_1=\left(1+\int_{\Omega}u^{(p+1) \beta_1} d x\right)^{\frac{1}{(p+1)\left(\beta_1-1\right)}}$ is bounded by \eqref{be}.
Now, we claim that there exists a constant $C_0>0$ independent of $n$, such that
\begin{equation}\label{5012}
B_{n+1} \leqslant \prod_{k=2}^{n+1} C_k^{\frac{1}{2\left(\beta_k-1\right)}} B_1 \leqslant C_0 B_1
    \end{equation}
We stress that once \eqref{5012} is established then by the Hölder inequality, we conclude that $u \in$ $L^r(\Omega)$, for every $r \in[1,+\infty)$. Further, by limiting the argument, we conclude that
$
\|u\|_{L^{\infty}(\Omega)} \leqslant C_0 B_1<+\infty.
$
 Now,  if $x\in \mathcal{N}$, then using the definition of $\mathcal{N}_s$ (see  \eqref{normal}) and $\mathcal{N}_s u(x)=0$, we have
$$ u(x)\int_{\Omega} \frac{dy}{|x-y|^{n+2s}}= \int_{\Omega} \frac{u(y) dy}{|x-y|^{n+2s}} \implies
u(x)=\frac{\int_{\Omega} \frac{u(y) dy}{|x-y|^{n+2s}}}{\int_{\Omega} \frac{dy}{|x-y|^{n+2s}}}.$$
Since $u\in L^{\infty}(\Omega)$ then we get $|u(x)|\leq \|u\|_{L^{\infty}(\Omega)}$, for each $x\in\mathcal{N}.$ Thus, we conclude that $u\in L^{\infty}(U).$
Thus, the proof of this theorem is now complete.

\section*{Acknowledgements} 
Tuhina Mukherjee acknowledges the financial support provided by CSIR-HRDG with sanction No. 25/0324/23/EMR-II. Lovelesh Sharma received assistance from the UGC Grant with reference no. 191620169606 funded by the Government
of India.


\begin{thebibliography}{10}
 \bibitem{MR3445279}
G.~Molica~Bisci, V.~D. Radulescu, and R.~Servadei.
\newblock {\em Variational methods for nonlocal fractional problems}, {\em Encyclopedia of Mathematics and its Applications}.
\newblock Cambridge University Press, Cambridge, (2016).


\bibitem{MR3651008}
S.~Dipierro, X.~Ros-Oton, and E.~Valdinoci.
\newblock Nonlocal problems with {N}eumann boundary conditions.
\newblock {\em Rev. Mat. Iberoam.}, 33(2), 377--416, (2017).


  \bibitem{biagi2021global}
 S. Biagi, M. Stefano, D. Mugnai and E. Vecchi.
  \newblock A {B}rezis-{O}swald approach for mixed local and nonlocal
              operators.
\newblock {\em Commun. Contemp. Math.}, 2(26), 2250057-28, (2024).
  

\bibitem{Rudin}
  W.~Rudin.
  \newblock {Real and Complex Analysis McGraw-Hill Book Company, New York, NY}.
  (1966).

 \bibitem{Kesavan}
     {S. Kesavan and V. Kalashnikov}.
  \newblock {Topics in functional analysis and applications},
    \newblock {\em Acta Appl. Math},
    37(3), (1994).

  \bibitem{MR3307955}
 {G. Franzina and G.  Palatucci}.
     \newblock {Fractional {$p$}-eigenvalues}.
   \newblock{Riv. Math. Univ. Parma }, 5(2), (2014).
   \bibitem{MR1962933}
 {E. Giusti}.
     \newblock {Direct methods in the calculus of variations}.
\newblock {\em World Scientific Publishing Co., Inc., River Edge, NJ},
     \newblock{981-238-043-4},
     (2003).

\bibitem{arora2023combined}
R. Arora and V. R{\u{a}}dulescu.
\newblock {Combined effects in mixed local--nonlocal stationary problems}.
\newblock{\em Proc. R. Soc. Edinb. A: Math.}, 1-47, (2023).

\bibitem {MR4444761}
 {P. Garain and A. Ukhlov}.
     \newblock{Mixed local and nonlocal {S}obolev inequalities with extremal
              and associated quasilinear singular elliptic problems}.
   \newblock{\em Nonlinear Anal. Theory Methods Appl},
              223, (2022).
              
    \bibitem{BS}
    \newblock{S. Biagi, S. Dipierro, E. Valdinoci and, E. Vecchi}.
   \newblock  {Mixed local and nonlocal elliptic operators: regularity and maximum principles}.
   \newblock{\em Commun. Partial Differ. Equ.}, 47(3), 585--629,
    (2022).
    
    \bibitem{MR2944369}
E.~Di~Nezza, G.~Palatucci, and E.~Valdinoci.
\newblock Hitchhiker's guide to the fractional {S}obolev spaces.
\newblock {\em Bull. Sci. Math.}, 136(5):521--573, (2012).

\bibitem{dipierro2022non}
S.~Dipierro, E.~P. Lippi, and E.~Valdinoci.
\newblock (Non) local logistic equations with neumann conditions.
\newblock {\em Ann. Inst. Henri Poincare (C)}, 40(5), 1093--1166, (2023).

\bibitem{MR4249816}
S.~Dipierro and E.~Valdinoci.
\newblock Description of an ecological niche for a mixed local/nonlocal dispersal: an evolution equation and a new {N}eumann condition arising from the superposition of {B}rownian and {L}\'{e}vy processes.
\newblock {\em Phys. A}, 575, 126052--20, (2021).

\bibitem{MR2924452}
C.Y. Kao, Y. Lou, and W.Shen.
\newblock Evolution of mixed dispersal in periodic environments.
\newblock {\em Discrete Contin. Dyn. Syst. Ser. B}, 17(6), 2047--2072, (2012).

\bibitem{MR3590678}
A.~Massaccesi and E.~Valdinoci.
\newblock Is a nonlocal diffusion strategy convenient for biological populations in competition?
\newblock {\em J. Math. Biol.}, 74(1-2):113--147, (2017).

\bibitem{MR3771424}
B.~Pellacci and G.~Verzini.
\newblock Best dispersal strategies in spatially heterogeneous environments: optimization of the principal eigenvalue for indefinite fractional {N}eumann problems.
\newblock {\em J. Math. Biol.}, 76(6), 1357--1386, (2018).

\bibitem{blazevski2013local}
D.~Blazevski and D.~del Castillo-Negrete.
\newblock Local and nonlocal anisotropic transport in reversed shear magnetic fields: Shearless cantori and nondiffusive transport.
\newblock {\em Physical Review E}, 87(6), 063106, (2013).


\bibitem{MR4225516}
G.~Pagnini and S.~Vitali.
\newblock Should {I} stay or should {I} go? {Z}ero-size jumps in random walks for {L}\'{e}vy flights.
\newblock {\em Fract. Calc. Appl. Anal.}, 24(1), 137--167, (2021).


\bibitem{MR1640881}
L.~C. Zheng and J.~C. He.
\newblock A class of singular nonlinear boundary value problems in the theory of pseudoplastic fluids.
\newblock {\em J. Northeast. Univ. Nat. Sci.}, 19(2), 208--211, (1998).

\bibitem{MR4065090}
B. Barrios and M. Medina.
\newblock Strong maximum principles for fractional elliptic and
              parabolic problems with mixed boundary conditions.
\newblock {\em Royal Soc. Edinburgh, Edinburgh}, 150(1), 475--495 (2020).

\bibitem{Stampacchia}
D. Kinderlehrer and G. Stampacchia.
\newblock An introduction to variational inequalities and their applications. {\em SIAM}, (2000).

\bibitem{MR3393266}
T. Leonori, I. Peral, and F. Soria.
\newblock Basic estimates for solutions of a class of nonlocal elliptic and parabolic equations.
\newblock {\em Discrete Contin. Dyn. Syst.}, 35(12), 6031--6068, (2015).

\bibitem{Garain}
P. Garain and A. Ukhlov.
\newblock Mixed local and nonlocal Sobolev inequalities with extremal and associated quasilinear singular elliptic problems. 
\newblock {\em Nonlinear Anal.}, 223, 113022, (2022).

\bibitem{Kinnunen}
P. Garain and J. Kinnunen.
\newblock On the regularity theory for mixed local and nonlocal quasilinear elliptic equations.
\newblock {\em Trans. Amer. Math. Soc.}, 375(08), 5393-5423, (2022).

\bibitem{Sharma}
T. Mukherjee and L. Sharma. 
\newblock On nonlocal problems with mixed operators and Dirichlet-Neumann mixed boundary conditions, arXiv:2311.02567, (2023).

\bibitem{LSharma}
T. Mukherjee, P. Pucci, and L. Sharma. 
\newblock Nonlocal critical exponent singular problems under mixed Dirichlet-Neumann boundary conditions. 
\newblock {\em J. Math. Anal. Appl.}, 531(2), 127843, (2023).

\bibitem{Castro}
A. Di Castro, T.  Kuusi, and G. Palatucci.
\newblock Local behavior of fractional p minimizers. 
\newblock {\em Ann. Inst. H. Poincaré Anal. Non Linéaire}, 33(5), 1279-1299, (2016).

\bibitem{Mugnai}
D. Mugnai, A.  Pinamonti, and E. Vecchi.  
\newblock Towards a Brezis–Oswald-type result for fractional problems with Robin boundary conditions. 
\newblock {\em Calc. Var. Partial Differ.}, 59, 1-25, (2020).

\bibitem{Leoni}
G. Leoni. 
\newblock A first course in fractional Sobolev spaces. 
\newblock {\em Amer. Math. Soc.,}  (2023).

\bibitem{Abdellaoui}
B. Abdellaoui, M. Medina, I. Peral, and A. Primo.
\newblock The effect of the Hardy potential in some Calderón–Zygmund properties for the fractional Laplacian.
\newblock {\em J. Differ. Equ.}, 260(11), 8160-8206, (2016).

\bibitem{Montoro}
A. Canino, L. Montoro, B. Sciunzi, and M. Squassina. 
\newblock Nonlocal problems with singular nonlinearity. \newblock {\em Bull. Sci. Math.}, 141(3), 223-250, (2017).
  \bibitem{Cassani}
D. Cassani, L. Vilasi, and Y. Wang.
  \newblock Local versus nonlocal elliptic equations: short-long range field interactions. 
  \newblock {\em Adv. Nonlinear Anal.}, 10(1), 895-921, (2020).
\bibitem{MR4391102}
S.~Biagi, S.~Dipierro, E.~Valdinoci, and E.~Vecchi.
\newblock A {H}ong-{K}rahn-{S}zeg\"{o} inequality for mixed local and nonlocal operators.
\newblock {\em Math. Eng.}, 5(1), (2023).
\bibitem{MR4357939}
C.~LaMao, S.~Huang, Q.~Tian, and C.~Huang.
\newblock Regularity results of solutions to elliptic equations involving mixed local and nonlocal operators.
\newblock {\em AIMS Math.}, 7(3):4199--4210, 2022.
\bibitem{arora2021combined}
R.~Arora and V.~D. Radulescu.
\newblock Combined effects in mixed local-nonlocal stationary problems.
\newblock {Proc. Roy. Soc. Edinburgh Sect. A}, 47(1), (2023).

\bibitem{MR4438596}
S.~Dipierro, E.~Proietti~Lippi, and E.~Valdinoci.
\newblock Linear theory for a mixed operator with {N}eumann conditions.
\newblock {\em Asymptot. Anal.}, 128(4), 571--594, (2022).

\bibitem{Barrios}
B. Barrios, I. De Bonis, M. Medina, and I. Peral.
\newblock Semilinear problems for the fractional laplacian with a singular nonlinearity. {\em Open Math. J.}, 13(1), (2015).

\bibitem{Bo}
L. Boccardo and L. Orsina.
\newblock Semilinear elliptic equations with singular nonlinearities. {\em Calc. Var. Partial Differ.}, 37(3-4), 363-380, (2010).

 \bibitem{Mukherjee}
 T. Mukherjee and L. Sharma.
 \newblock On elliptic problems with mixed operators and Dirichlet-Neumann
boundary conditions.  arXiv:2311.02567.

\bibitem{JG}
J. Giacomoni,  T. Mukherjee and L. Sharma.
\newblock On an eigenvalue problem associated with mixed operators under
mixed boundary conditions. arXiv:2411.16499.

\bibitem{PG}
P. Garain. 
\newblock On a class of mixed local and nonlocal semilinear elliptic equation with singular nonlinearity. \newblock {\em J. Geom. Anal.}, 33(7), 212, (2023).

\bibitem{RA}
M. G. Crandall,P. H. Rabinowitz, and L. Tartar. 
\newblock On a Dirichlet problem with a singular nonlinearity. {\em Commun. Partial Differ. Equ.}, 2(2), 193-222, (1977).

\bibitem{LA}
A. C. Lazer and P. J. McKenna. 
\newblock On a singular nonlinear elliptic boundary-value problem. {\em Proc. Am. Math. Soc.}, 111(3), 721-730, (1991).

\bibitem{Su}
X. Su, E. Valdinoci, Y. Wei, and J. Zhang.  
Regularity results for solutions of mixed local and nonlocal elliptic equations. {\em MATH Z }, 302(3), 1855--1878, (2022).

\bibitem{JDE}
X. Su, E. Valdinoci, Y. Wei, and J. Zhang.  
On some regularity properties of mixed local and nonlocal elliptic equations. {\em J. Differ. Equ.}, 416, 576--613, (2024).

\bibitem{Tm}
J. Giacomoni, T. Mukherjee, and L. Sharma.  On an eigenvalue problem associated with mixed operators under mixed boundary conditions. Preprint available on arXiv:2411.16499, (2024).

\bibitem{MR2099611}
N.~Hirano, C.~Saccon, and N.~Shioji.
\newblock Existence of multiple positive solutions for singular elliptic problems with concave and convex nonlinearities.
\newblock {\em Adv. Differential Equations}, 9(1-2), 197--220, (2004).

\bibitem{MR1964476}
Y.~Haitao.
\newblock Multiplicity and asymptotic behavior of positive solutions for a singular semilinear elliptic problem.
\newblock {\em J. Differential Equations}, 189(2), 487--512, (2003).

\bibitem{MR2446183}
N.~Hirano, C.~Saccon, and N.~Shioji.
\newblock Brezis-{N}irenberg type theorems and multiplicity of positive solutions for a singular elliptic problem.
\newblock {\em J. Differential Equations}, 245(8), 1997--2037, (2008).

\bibitem{SK}
S. Kesavan. 
\newblock Topics in functional analysis and applications, (1989).

\end{thebibliography}

\Addresses
\end{document}